\documentclass[twoside]{amsart}

\usepackage{amssymb}
\usepackage{graphicx}
\usepackage{caption}
\usepackage{bm}

\newcommand{\rmd}{\mathrm{d}}
\newcommand{\rmi}{\mathrm{i}}

\newcommand{\I}{\mathcal{I}}
\newcommand{\J}{\mathcal{J}}
\newcommand{\K}{\mathcal{K}}
\newcommand{\M}{\phantom{-}}

\newcommand{\kFr}{\mathfrak{k}}
\newcommand{\lFr}{\mathfrak{l}}
\newcommand{\mFr}{\mathfrak{m}}
\newcommand{\nFr}{\mathfrak{n}}

\DeclareMathOperator{\modR}{mod}

\newcommand{\Jac}{\mathrm{Jac}}

\DeclareMathOperator{\rank}{rank}

\DeclareMathOperator{\Complex}{\mathbb{C}}
\DeclareMathOperator{\Integer}{\mathbb{Z}}

\DeclareMathOperator{\wgt}{wgt}

\newtheorem{lemma}{Lemma}
\newtheorem{teo}{Theorem}
\newtheorem{prop}{Proposition}
\newtheorem{cor}{Corollary}
\newtheorem{conj}{Conjecture}

\theoremstyle{definition}
\newtheorem{rem}{Remark}
\newtheorem{exam}{Example}

\theoremstyle{plain}

\title[General derivative Thomae formula]
{General derivative Thomae formula \\ for singular half-periods}
\author{J Bernatska}
\address{}
\email{jbernatska@gmail.com, bernatska.julia@ukma.edu.ua}
\date{\today}

\allowdisplaybreaks[4]

\begin{document}
 
\maketitle 
\begin{abstract}
The paper develops second Thomae theorem in hyperelliptic case. 
The main formula, called general Thomae formula, provides expressions for
values at zero of the lowest non-vanishing derivatives of theta functions 
with singular characteristics of arbitrary multiplicity in terms of branch points and period matrix. 
We call these values derivative theta constants.
First and second Thomae formulas follow as particular cases.

Some further results are derived. 
Matrices of second derivative theta constants 
(Hessian matrices of zero-values of theta functions with characteristics
of multiplicity two) have rank three in any genus.
Similar result about the structure of order three tensor of third derivative theta constants
is obtained, and a conjecture regarding higher multiplicities is made.
As a byproduct a generalization of Bolza formulas are deduced.
\end{abstract}


\section{Introduction}
Thomae formulas are of great interest in many areas of mathematics and physics such as 
quantum field theory, string theory, theory of integrable systems, number theory, $p$-adic analysis etc.
This paper provides a development of the classical work of Thomae \cite{Tho870}.
First and second Thomae formulas give a representation of theta
constants with non-singular even and odd characteristics in terms of branch points and period matrix
of a hyperelliptic Riemann surface.
In the present paper a similar representation of derivative theta constants with singular characteristics 
is found for all possible multiplicities.
This problem was not considered in mathematical literature since the time of Thomae.

Instead, generalizations of Thomae formulas to the case of $\Integer_N$-curves, 
also called cyclic covers of $\mathbb{CP}^{1}$ or simply cyclic curves,  were discovered. 
This was initiated by paper \cite{BR1988}, where a generalization of first Thomae formula
provided an expression for determinant of Dirac's operator in terms of branch points of $\Integer_N$-curve.
This gave rise to a flow of publications. 
In \cite{Nak1997} a rigorous proof of the mathematical result from \cite{BR1988} was given.  
Then first Thomae formula was generalized to a special class of singular $\Integer_N$-curve in \cite{EG2006}, 
to general cyclic covers of $\mathbb{CP}^{1}$ in \cite{Kop2010}, 
to Abelian covers of $\mathbb{CP}^{1}$ in \cite{KZ2019}, 
and developed in other papers and a book of H.\;Farkas and Sh.\;Zemel 
Generalization of Thomae's Formula for $\Integer_N$ Curves (2010) containing many examples.
A detailed generalization of first Thomae formula for a trigonal cyclic curve with a specific choice of 
symplectic cohomology basis is given in \cite{MT2010}.
The only generalization of second Thomae formula was obtained in \cite{EKZ} for trigonal cyclic curves.

This paper is organized as follows. Section 2 contains the minimal background, definitions and notation
regarding theta and sigma functions, Thomae theorems and some auxiliary lemmas. 
Section 3 is devoted to the main theorem (Theorem~\ref{T:ThN}), which generalizes second Thomae theorem,  
with a detailed proof. In Section 4 corollaries of the main theorem are presented with examples in genera $3$, $4$, $5$ and~$6$.
Corollaries~\ref{C:thomae2I}, \ref{C:thomaeK34}, \ref{C:thomaeK56} provide some further representations of first, 
second and third derivative theta constants. This allows to establish some essential results. In particular, the rank of matrices of 
second derivative theta constants (Hessian matrices) is three in any genus (Theorem~\ref{T:DetD2theta}).
Also a generalization of Bolza formulas \cite{Bolza} for branch points and their symmetric fuctions is obtained.
Section 5 contains a summary of our results and links to related problems posed in the literature.

\section{Preliminaries}
\subsection{Hyperelliptic curve}
A hyperelliptic curve $\mathcal{C}$ is described by its branch points $\{(e_j,0)\}$, and defined by the equation
\begin{gather}\label{Cg}
 0 = f(x,y) = -y^2 + \prod_{j=1}^{2g+1} (x-e_j).
\end{gather}
The branch points are all distinct if a curve is non-degenerate. 
One more branch point is located at infinity which serves as a base point.

Homology basis is defined after H.\,Baker \cite[p.\,303]{bak898}.
One can imagine a continuous path through all branch points, which ends at infinity.
The branch points are denoted by $\{e_j\}_{j=1}^{2g+1}$ along the path, infinity is denoted by $e_{2g+2}$.
Cuts are made between points $e_{2k-1}$ and $e_{2k}$ with $k$ from $1$ to $g+1$. 
Canonical homology cycles are defined as follows.
Each $\mathfrak{a}_k$-cycle encircles the cut $(e_{2k-1},e_{2k})$, $k=1$, \ldots $g$,
and each $\mathfrak{b}_k$-cycle enters two cuts $(e_{2k-1},e_{2k})$ and $(e_{2g+1},e_{2g+2})$,
see fig.~\ref{cycles}.
\begin{figure}[h]
\includegraphics[width=0.9\textwidth]{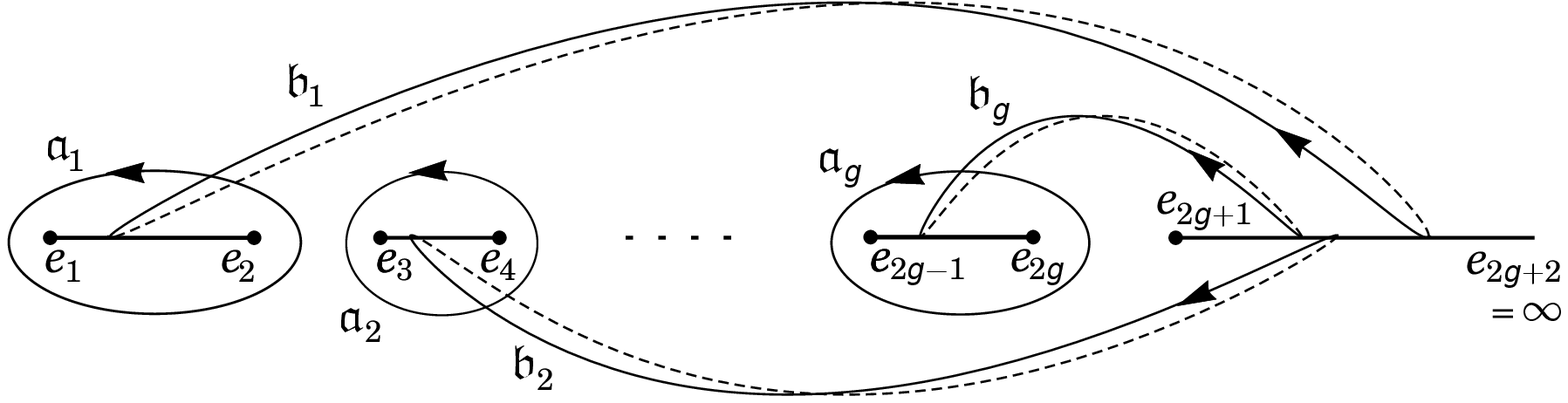}
\caption{} \label{cycles}
\end{figure}

The standard cohomology basis is employed, it consists of 
first kind differentials $\rmd u = (\rmd u_1$, $\rmd u_3$, $\dots$, $\rmd u_{2g-1})^t$
and second kind differentials $\rmd r = (\rmd r_1$, $\rmd r_3$, $\dots$, $\rmd r_{2g-1})^t$ 
associated to the first kind differentials, see for example \cite[p.\,306]{bak898},
\begin{align}
&\rmd u_{2n-1} = \frac{x^{g-n} \rmd x}{\partial_y f},\qquad n=1,\,\dots,\,g,  
\label{HDifCg}\\
&\rmd r_{2n-1} = \frac{\rmd x}{\partial_y f} \sum_{k=1}^{2n-1} k \lambda_{4n-2k+2} x^{g-n+k},
\qquad n=1,\,\dots,\,g.  \label{MDifCg}
\end{align} 
Here $\lambda_l$ denote coefficients of the curve \eqref{Cg} written in the form
\begin{gather*}
 f(x,y)=-y^2 + x^{2g+1} + \sum_{i=0}^{2g} \lambda_{4g+2-2i} x^i.
\end{gather*}
The differentials are labeled by Sat\={o} weights in subscripts, for convenience. The weight shows
an exponent of the leading term in expansion of the corresponding integral 
about infinity in parameter $\xi$, $x=\xi^{-2}$, namely $\wgt u_{2n-1} = 2n-1$, and $\wgt r_{2n-1} = -(2n-1)$.

Integrals of the differentials along the canonical homology cycles give
first and second kind periods
\begin{align*}
 &\omega = (\omega_{ij})= \bigg( \int_{\mathfrak{a}_j} \rmd u_i\bigg),&
 &\omega' = (\omega'_{ij}) = \bigg(\int_{\mathfrak{b}_j} \rmd u_i \bigg),& \\
 &\eta = (\eta_{ij}) = \bigg(\int_{\mathfrak{a}_j} \rmd r_i \bigg),&
 &\eta' = (\eta'_{ij}) = \bigg(\int_{\mathfrak{b}_j} \rmd r_i \bigg).&
\end{align*}
The $g\times g$ matrices $\omega,\omega', \eta, \eta'$ form $2g\times 2g$ matrix
\begin{gather*}
 \Omega = \begin{pmatrix} \omega & \omega' \\
           \eta & \eta'  \end{pmatrix},
\end{gather*}
which is symplectic with respect to a complex structure $J$,
$J^2=-1_{2g}$, where $1_{2g}$ is the identity matrix of size $2g$, and $J^t=-J$,
\begin{gather}\label{LegendreIdnt}
\Omega J \Omega^t = 2\pi \rmi J.
\end{gather}
Relation \eqref{LegendreIdnt} is known as the generalized Legendre relation, see for example \cite[p.\,16]{bel12}. 
With complex structure matrix $J$ of the form
\begin{gather*}
J = \left( \begin{array}{cc}  \M 0_g & 1_g \\ -1_g & 0_g
\end{array}  \right),
\end{gather*}
where $1_g$ is the identity matrix of size $g$, and $0_g$ is $g\times g$ zero matrix.
\eqref{LegendreIdnt} produces relations
\begin{align}\label{LegRel}
 &\begin{array}{l}\omega \omega'{}^t = \omega' \omega^t,\\
 \eta \eta'{}^t = \eta' \eta^t, \\
 \eta \omega'{}^t = \eta' \omega^t - 2\pi \rmi 1_g, \end{array} & &\text{or}&
 &\begin{array}{l} \omega^t \eta = \eta^t \omega,\\
 \omega'{}^t \eta' = \eta'{}^t \omega', \\
 \omega'{}^t\eta = \eta'{}^t \omega - 2\pi \rmi 1_g.
 \end{array} &
\end{align}

\subsection{Theta and sigma functions}
Each curve of the family $\mathcal{C}$ has a Jacobian variety $\Jac(\mathcal{C})=\Complex^g\backslash \mathfrak{P}$,
which is a quotient space of $\Complex^g$ by the lattice $\mathfrak{P}$ 
of periods formed by columns of the matrix $(\omega,\omega')$.
Let $u$ with coordinates $(u_1,$ $u_3$, \ldots, $u_{2g-1})^t$ denote a point of Jacobian.
These are variables of sigma function, which is defined here with the help of theta function. 

Theta function is an entire function on $\Complex^g$ 
defined with respect to normalized periods $(1_g,\tau)$, where $\tau = \omega^{-1}\omega'$
is a symmetric matrix with positive imaginary part: $\tau^t=\tau$, $\Im \tau >0$,
that is $\tau$ belongs to Siegel upper half-space. 
Normalized holomorphic differentials are
\begin{gather*}
 \rmd v = \omega^{-1} \rmd u,
\end{gather*}
and similarly normalized coordinates of the Jacobian are defined: $v=\omega^{-1} u$,
$v=(v_1,v_2,\dots,v_g)^t$. This change of coordinates is essential for the relation \eqref{SigmaDef} between
theta and sigma functions.
Riemann theta function is defined for $v\in \Complex^g$ as a Fourier series of the form
\begin{gather}\label{ThetaDef}
 \theta(v;\tau) = \sum_{n\in \Integer^g} \exp \big(\rmi \pi n^t \tau n + 2\rmi \pi n^t v\big).
\end{gather}
Abel's map $\mathcal{A}$ maps the curve to its Jacobian
\begin{gather}\label{AbelM}
 \Jac(\mathcal{C}) \ni \mathcal{A}(P) = \int_{\infty}^P \rmd v,\qquad P\in \mathcal{C}.
\end{gather}
Abel's map of a positive divisor $\mathcal{D} = \sum_{i=1}^n P_i$ 
on $\mathcal{C}$ is defined by
\begin{gather}\label{AbelMD}
 \mathcal{A}(\mathcal{D}) = \sum_{i=1}^n 
 \int_{\infty}^{P_i} \rmd v.
\end{gather}

Each branch point $(e,0)$ of a hyperelliptic curve \eqref{Cg} is identified with a half-period,
see \cite[\S\,202 p.\,300-301]{bak897}
\begin{gather}\label{AbelMBranchP}
 \mathcal{A}(e) = \int_{\infty}^{(e,0)} \rmd v = \varepsilon/2 + \tau \varepsilon'/2,
 \qquad \begin{bmatrix} \varepsilon'{}^t \\ 
\varepsilon^t \end{bmatrix} = [\varepsilon],
\end{gather}
where components of $\varepsilon$ and $\varepsilon'$ are $0$ or $1$.
The integer $2\times g$-matrix $[\varepsilon]$ is the characteristic of branch point $e$.
Characteristics are added by the rule $[\varepsilon]+[\delta] = ([\varepsilon]+[\delta]) \modR 2$.
Theta function with characteristic $[\varepsilon]$ is given by the formula
\begin{multline}\label{ThetaDef}
 \theta[\varepsilon](v;\tau) = \exp\big(\rmi \pi (\varepsilon'{}^t/2) \tau (\varepsilon'/2)
 + 2\rmi \pi  (v+\varepsilon/2)^t \varepsilon'/2\big) \times \\ \times \theta(v+\varepsilon/2 + \tau \varepsilon'/2;\tau).
\end{multline}
A characteristic $[\varepsilon]$ is odd whenever $\varepsilon^t \varepsilon' \modR 2 = 0$, 
and even whenever $\varepsilon^t \varepsilon' \modR 2 = 1$. Theta function with characteristic
has the same parity as its characteristic.

Another entire function, sigma function, is connected to theta function by relation
\begin{gather}\label{SigmaDef}
 \sigma(u) = \frac{1}{C} \exp\Big({-}\frac{1}{2} u^t \varkappa u \Big) \theta[K](\omega^{-1} u; \omega^{-1}\omega'),
\end{gather}
where $\varkappa = \eta \omega^{-1}$ is a symmetric matrix defined through period matrices $\omega$ and $\eta$,
and $[K]$ denotes the characteristic of vector of Riemann constants.
This formula arises as a definition of fundamental sigma functions in \cite[p.\,97]{bak907} without constant $C$.
In \cite{Gra1988} this constant was found by methods developed in \cite{bak907,bak898}. 
Also the constant is explicitly defined in \cite[Eq.\,(3.29) p.\,906]{EHKKL2011},  
which is proved in \cite[p.\,33]{bel12}.

\subsection{Characteristics in hyperelliptic case}\label{ss:CharHyper}
The method of constructing characteristics in hyperelliptic case is adopted from \cite[p.\,1012]{ER2008}.
It is based on the definition \eqref{AbelMBranchP} of half-period characteristics.
Let $[\varepsilon_k]$ be the characteristic of branch point $e_k$.
Evidently, $[\varepsilon_{2g+2}]=0$. Guided by the picture of canonical homology cycles, one can find
\begin{align*}
 &\mathcal{A}(e_{2g+1}) = \mathcal{A}(e_{2g+2}) + \sum_{k=1}^g \int_{e_{2k-1}}^{e_{2k}} \rmd v &
 &[\varepsilon_{2g+1}] = \big[ {}^{00\dots 00}_{11\dots 11} \big], & \\
 &\mathcal{A}(e_{2g}) = \mathcal{A}(e_{2g+1}) - \int_{e_{2g}}^{e_{2g+1}} \rmd v &
 &[\varepsilon_{2g}] = \big[ {}^{00\dots 01}_{11\dots 11} \big], & \\
 &\mathcal{A}(e_{2g-1}) = \mathcal{A}(e_{2g}) - \int^{e_{2g}}_{e_{2g-1}} \rmd v&
 &[\varepsilon_{2g+1}] = \big[ {}^{00\dots 01}_{11\dots 10} \big], &
\intertext{for $k$ from $g-1$ to $2$}
 &\mathcal{A}(e_{2k}) = \mathcal{A}(e_{2k+1}) - \int_{e_{2k}}^{e_{2k+1}} \rmd v &
 &[\varepsilon_{2k}] = \big[ \overbrace{{}^{00\dots 0}_{11\dots 1}}^{k-1}\!{}^{10 \dots 0}_{10 \dots 0} \big], & \\
 &\mathcal{A}(e_{2k-1}) = \mathcal{A}(e_{2k}) - \int^{e_{2k}}_{e_{2k-1}} \rmd v&
 &[\varepsilon_{2k-1}] = \big[ \overbrace{{}^{00\dots 0}_{11\dots 1}}^{k-1}\!{}^{10 \dots 0}_{00 \dots 0} \big], &
\intertext{and finally}
 &\mathcal{A}(e_{2}) = \mathcal{A}(e_{3}) - \int_{e_{2}}^{e_{3}} \rmd v &
 &[\varepsilon_{2}] = \big[ {}^{10\dots 0}_{10\dots 0} \big], & \\
 &\mathcal{A}(e_{1}) = \mathcal{A}(e_{2}) - \int^{e_{2}}_{e_{1}} \rmd v &
 &[\varepsilon_{1}] = \big[ {}^{10\dots 0}_{00\dots 0} \big]. &
 \end{align*}
This set of characteristics is azygetic and serves as a fundamental system, see \cite[pp.\,181--184]{RF1974}.

Characteristic $[K]$ of the vector of Riemann constants $K$ equals 
the sum of all odd characteristics of branch points, there are $g$ such characteristics,
see \cite[\S\,200 p.\,297, \S\,202 p.\,301]{bak897}.
Actually,
\begin{gather*}
 [K] = \sum_{k=1}^g [\varepsilon_{2k}].
\end{gather*}

\subsection{Characteristics and partitions}
Let $\mathcal{I}\cup \J$  be a partition of the set of indices of all branch points $\{1,2,\dots, 2g+2\}$,
denote by $[\varepsilon(\I)] = \sum_{i\in\I} [\varepsilon_i]$ the characteristic of
\begin{gather*}
 \mathcal{A} (\I) = \sum_{i\in\I} \mathcal{A}(e_i) = \frac{1}{2} \varepsilon_\I  + \frac{1}{2} \tau \varepsilon'_\I.
\end{gather*}
Below a partition is often referred to by the part of less cardinality, denoted by $\I$.

Characteristics of $2g+2$ branch points of \eqref{Cg} serve as a basis for constructing
all $2^{2g}$ half-period characteristics.
According to \cite[p.\,13]{fay973} and \cite[\S\,202 p.\,301]{bak897} 
all half-period characteristics are represented by 
partitions of $2g+2$ indices of the form $\I_m\cup \J_m$ with $\I_m = \{i_1,\,\dots,\, i_{g+1-2m}\}$
and $\J_m = \{j_1,\,\dots,\, j_{g+1+2m}\}$, where $m$ runs from $0$ to $[(g+1)/2]$, and $[\cdot]$ means the integer part. 
Number $m$ is called \emph{multiplicity}.
Index $2g+2$ corresponding to infinity is usually omitted in the sets, and inferred in the part with an incomplete number of indices. 

Introduce also characteristic $[\I_m] = [\varepsilon(\I_m)] + [K]$ of
\begin{gather*}
 \sum_{i\in\I_m} \mathcal{A}(e_i) + K = \frac{1}{2}\delta_{\I_m} + \frac{1}{2}\tau \delta'_{\I_m},
\end{gather*}
which  corresponds to a partition $\I_m\cup \J_m$. 
Note that $[\J_m]$ represents the same characteristic as $[\I_m]$.
Characteristics $[\I_m]$ of even multiplicity $m$ are even, and of odd $m$ are odd.
According to Riemann theorem $\theta(v+\mathcal{A} (\I_m)+K)$ vanishes to order $m$ at $v=0$.
Characteristics of multiplicity $0$ are called \emph{non-singular even characteristics},
there are $\binom{2g+1}{g}$ such characteristics.
There exist $\binom{2g+2}{g-1}$ characteristics of multiplicity $1$, which are called \emph{non-singular odd}.
All other characteristics are called \emph{singular}.
The number of characteristics of multiplicity $m > 1$ is $\binom{2g+2}{g+1-2m}$.

Characteristic $[K]$ corresponds to the partition $\{\}\cup\{1,2,\dots,2g+1\}$, which is always unique, and
$\theta[K](v)$ vanishes to the maximal order $[(g+1)/2]$ at $v=0$.
In what follows, representation of characteristics in terms of partitions is preferable, 
because this makes clear which order of vanishing a theta function has at $v=0$.

Let a collection of branch points $\{e_i\mid i\in \mathcal{I}\}$ 
correspond to a partition $\I \cup \J$.
Then $s_n(\I)$ denotes an elementary symmetric polynomial of degree $n$ in  $\{e_i\mid i\in\I\}$,
and $\Delta(\I)$ denotes the Vandermonde determinant in branch points from the collection
\begin{gather*}
 \Delta(\I) = \prod_{\substack{i > l \\ i,l \in \I}} (e_i-e_l).
\end{gather*}
The Vandermonde determinant in all branch points of the curve is denoted by $\Delta$.

\begin{rem}
 Let branch points in all factors $(e_i-e_l)$ be ordered in such a way that $i>l$, 
 we call this \emph{right ordering}.
 This allows to avoid multiplier $\epsilon$, which arises in many relations. 
 Such ordering was suggested by Baker \cite[p.\,346]{bak898}.
\end{rem}

Theta function with characteristic $[\I]$ corresponds to sigma function at a half-period $\mathcal{A}(\I)$ 
as defined in \cite{EHKKL2011,EHKKLS}, indeed 
\begin{multline}\label{SigmaCh}
 \sigma(\omega \mathcal{A}(\I))  = \frac{1}{C} \exp\Big({-}\frac{1}{2} \mathcal{A}(\I)^t \omega^t \eta \mathcal{A}(\I) \Big) 
 \theta[\varepsilon_K]\big(\mathcal{A}(\I); \omega^{-1}\omega'\big) \\
 = \frac{1}{C} \exp \Big({-}\frac{1}{8} \big(\omega \varepsilon + \omega' \varepsilon'\big)
 \big(\eta \varepsilon + \eta' \varepsilon'\big)\Big) \theta[\I]\big(0; \omega^{-1}\omega'\big)
\end{multline}
where $\varepsilon$, $\varepsilon'$ denote $\varepsilon(\I)$, $\varepsilon'(\I)$, 
and Legendre relations \eqref{LegRel} are used
\begin{gather*}
\tau \omega^t \eta = \omega^{-1} \omega' \omega^t \eta =  \omega'{}^t \eta
= \eta'{}^t \omega - 2\pi \rmi 1_g.
\end{gather*}
One can notice that exponential factor in \eqref{SigmaCh} is $\exp\big({-}\frac{1}{8}\tilde{\varepsilon}^t \tilde{\varepsilon}'\big)$
with the characteristic $[\varepsilon]$ transformed by symplectic matrix $\Omega$
\begin{gather*}
 \begin{pmatrix} \tilde{\varepsilon} \\ \tilde{\varepsilon}' \end{pmatrix} =
 \Omega \begin{pmatrix} \varepsilon \\ \varepsilon' \end{pmatrix}.
\end{gather*}

\subsection{Notation of theta constants}
In what follows we use notation $\partial_{v_i}$ for derivative with respect to variable $v_i$, and 
omit argument $\tau$ of theta function, so 
\begin{align*}
 &\partial_{v_i} \theta[\varepsilon](v) = \frac{\partial}{\partial v_i}  \theta[\varepsilon](v;\tau),&\\
 &\partial^2_{v_i,v_j} \theta[\varepsilon](v) 
= \frac{\partial^2}{\partial v_i\partial v_j}  \theta[\varepsilon](v;\tau),&\\
&\text{etc.}&
\end{align*}
We also use the standard notation $\theta[\varepsilon]$ for theta constant with characteristic $[\varepsilon]$, 
or $\theta[\I_0]$ for one with the characteristic corresponding to partition $\I_0\cup\J_0$.
Thus
\begin{align*}
 &\theta[\I_0] = \theta[\I_0](0;\tau),&\\
 &\partial_{v_i} \theta[\I_1] = \frac{\partial}{\partial v_i}  \theta[\I_1](v;\tau)\big|_{v=0},&\\
 &\partial^2_{v_i,v_j} \theta[\I_2] 
 = \frac{\partial^2}{\partial v_i\partial v_j}  \theta[\I_2](v;\tau)\big|_{v=0}	,&\\
&\text{etc.}&
\end{align*}

In the literature we find values at $v=0$ of theta functions 
with non-singular even characteristics, which are called theta constants, 
or theta constants of the first kind \cite[p.1011]{ER2008}. 
Ibid values at $v=0$ of first derivatives of theta functions with non-singular odd characteristics
are called theta constants of the second kind. 
However, there is no consensual term for the latter, 
and no term at all for the values of higher derivatives of theta functions at $v=0$.

As mentioned above, $\theta[\I_m](v)$ vanishes to order $m$ at $v=0$.
We are interested in values at $v=0$ of the lowest non-vanishing derivatives of
theta functions with characteristics of arbitrary multiplicity $m$, 
that is $\partial_v^m \theta[\I_m]$. Here symbol $\partial_v^m$  denotes a tensor of order $m$ 
of partial $m$-th derivatives with respect to all combinations constructed from $g$ components of $v$.
We call them \emph{derivative theta constants}, or \emph{$m$-th derivative theta constants}
more precisely.

\subsection{First Thomae Formula}
\newtheorem*{FTteo}{First Thomae theorem}
\begin{FTteo}
 Let $\I_0\cup \J_0$ with $\I_0=\{i_1$, \ldots, $i_{g}\}$ 
 and $\J_0 = \{j_1,\,\dots,\,j_{g+1}\}$  be a partition
 of the set $\{1,2,\dots,2g+1\}$ of indices of finite branch points, and  $[\I_0]$ denotes 
 the non-singular even characteristic corresponding to $\mathcal{A} (\I_0) + K$.  Then
 \begin{gather}\label{thomae1}
 \theta[\I_0] = \epsilon \bigg(\frac{\det \omega}{\pi^g}\bigg)^{1/2} \Delta(\I_0)^{1/4} \Delta(\J_0)^{1/4}.
\end{gather}
where $\epsilon$ satisfies $\epsilon^8=1$,
and $\Delta(\I_0)$, $\Delta(\J_0)$ denote Vandermonde determinants 
built from $\{e_i\mid i\in \I_0\}$ and $\{e_j\mid j\in \J_0\}$.
\end{FTteo}
For a proof see \cite[Proposition\;3.6, p.46]{fay973}.
This form of first Thomae theorem is taken from \cite[p.\,1014]{ER2008}, 
as well as second Thomae theorem \cite[p.\,1015]{ER2008} below.

\subsection{Second Thomae Formula}
\newtheorem*{STteo}{Second Thomae theorem}
\begin{STteo}
 Let $\I_1\cup \J_1$ with $\I_1=\{i_1$, \ldots, $i_{g-1}\}$ and $\J_1 = \{j_1,\,\dots,\,j_{g+2}\}$ 
 be a partition of the set $\{1,2,\dots,2g+1\}$ of indices of finite branch points, and
 $[\I_1]$ denote the non-singular odd characteristic  
 corresponding to $\mathcal{A} (\I_1) + K$. Then for any $n \in \{1, \dots, g\}$
 \begin{multline}\label{thomae2}
 \frac{\partial }{\partial v_n} \theta[\I_1](v)\big|_{v=0} \\
 = \epsilon \bigg(\frac{\det \omega}{\pi^g}\bigg)^{1/2} \Delta(\I_1)^{1/4} \Delta(\J_1)^{1/4} 
 \sum_{j=1}^g  (-1)^{j-1} s_{j-1}(\I_1) \omega_{j,n}.
\end{multline}
where  $\epsilon$ satisfies $\epsilon^8=1$, and $\Delta(\I_1)$, $\Delta(\J_1)$ denote
Vandermonde determinants built from $\{e_i\mid i\in \I_1\}$ and $\{e_j\mid j\in \J_1\}$,
then $s_{j}(\I)$ denotes the elementary symmetric polynomial of degree $j$ in $\{e_i\mid i\in \I\}$.
\end{STteo}

This result is nicely presented in a matrix form 
\begin{multline}\label{thomae2MF}
 \begin{pmatrix} \partial_{v_1} \\ \partial_{v_2}  \\ \vdots \\ \partial_{v_g}  \end{pmatrix} 
 \theta[\I_1] (v)\big|_{v=0} \\
 = \epsilon \bigg(\frac{\det \omega}{\pi^g}\bigg)^{1/2} \Delta(\I_1)^{1/4} \Delta(\J_1)^{1/4}
 \omega^t  \begin{pmatrix} s_0(\I_1) \\ - s_1(\I_1) \\ \vdots \\ (-1)^{g-1} s_{g-1}(\I_1) \end{pmatrix}.
\end{multline}
In terms of non-normalized variables $u$ 
\begin{multline}\label{thomae2MFu}
 \begin{pmatrix} \partial_{u_1} \\ \partial_{u_3}  \\ \vdots \\ \partial_{u_{2g-1}}  \end{pmatrix}  
 \theta[\I_1] (\omega^{-1} u )\big|_{u=0} \\
 =  \epsilon \bigg(\frac{\det \omega}{\pi^g}\bigg)^{1/2} \Delta(\I_1)^{1/4} \Delta(\J_1)^{1/4}
 \begin{pmatrix} s_0(\I_1) \\ - s_1(\I_1) \\ \vdots \\ (-1)^{g-1} s_{g-1}(\I_1) \end{pmatrix}.
\end{multline}

Second Thomae theorem is derived from Lemma~\ref{L:thetaquotient}, for more detail see the proof invented by V.\,Enolski, published in \cite[pp.\,92--95]{Eilers2018} (Enolski supervised K.\,Eilers' research) and \cite[pp.\,2--4]{EKZ}.

Here the definition of elementary symmetric polynomials $s_l$
in $\{e_i\mid i\geqslant 0\}$ is  recalled
\begin{gather*}
\sum_{l \geqslant 0} t^l s_l = \prod_{i\geqslant 0} (1+e_i t).
\end{gather*}

\subsection{Auxiliary Lemmas }\label{Thomae2Proof}
The following Lemmas are essential in the proof of second Thomae formula. 
They are used below as well.
\begin{lemma}\label{L:thetaquotient}
Let $\mathcal{D}= \sum_{i=1}^g P_i$ be a divisor of $g$ finite points $\{P_i=(x_i,y_i)\}$ 
of a hyperelliptic curve of genus $g$,
and $[\varepsilon_k]$ denote the characteristic of a branch point $e_k$. Then 
 \begin{gather}\label{thetaquotient}
\frac{\theta[\varepsilon_k] \big(v(\mathcal{D}) + K\big)^2}{\theta \big(v(\mathcal{D}) + K\big)^2}
= \frac{\epsilon}{\sqrt{\partial_x f(e_k,0)}}\prod_{i=1}^g(e_k-x_i),
\end{gather}
where $\epsilon^4=1$, and $K$ denotes the Riemann constant vector related to the base point at infinity.
\end{lemma}
Here and below $v(\mathcal{D})$ denotes the argument of theta functions
which is considered as a function of a divisor $\mathcal{D}$.

A proof can be found in \cite[pp.\,92--95]{Eilers2018} or \cite[pp.\,2--4]{EKZ}.

\begin{lemma}\label{L:InvJac}
The Jacobian matrix $\frac{\partial(x_1,x_2,\dots,x_g)}{\partial(v_1,v_2,\dots,v_g)}$ 
 of Abel's map $\mathcal{A}: \mathcal{C}^g \mapsto \Jac(\mathcal{C})$ has the following entries
\begin{gather*}
 \frac{\partial x_p}{\partial v_n} = 
  -2y_p \frac{\sum_{j=1}^g(-1)^{j-1} s_{j-1}^{(p)} \omega_{jn}}{\prod_{\substack{l\neq p \\ l=1}}^g (x_p-x_l)},
\end{gather*}
where $p$ and $n$ run from $1$ to $g$, and $x_p$ represents a points $P_p = (x_p,y_p)$ of the curve.
\end{lemma}

\begin{proof}
Introduce
\begin{align}
 &\mathcal{P}_p(x) = \prod_{\substack{l\neq p \\ l=1}}^g \frac{x-x_l}{x_p-x_l} 
 = \frac{\sum_{j=1}^{g} (-1)^{j-1} s_{j-1}^{(p)} x^{g-j} }{\prod_{\substack{l\neq p \\ l=1}}^g (x_p-x_l)},
 \quad p=1,\dots,g, \label{Pjdef}
\end{align} 
where $s_{l}^{(p)}$  denotes the elementary symmetric polynomial of order $l$ built from elements 
$\{x_1,\dots,\,x_g\}\backslash\{x_p\}$. Evidently,
\begin{gather*}
 \mathcal{P}_p(x_k) = \delta_{pk}.
\end{gather*}

Taking into account that $v=\omega^{-1} u$, where $u$ is defined by \eqref{HDifCg},
we find the inverse to Jacobian matrix 
$\frac{\partial (x_1,\dots x_g)}{\partial (u_1,\dots,u_{2g-1})} = J$ in non-normalized variables
\begin{gather*}
 J^{-1} = \begin{pmatrix} \frac{x_1^{g-1}}{-2y_1} & \frac{x_2^{g-1}}{-2y_2} & \cdots & \frac{x_g^{g-1}}{-2y_g}\\ 
 \vdots & \vdots & \ddots & \vdots \\ \frac{x_1}{-2y_1} & \frac{x_2}{-2y_2} & \cdots & \frac{x_g}{-2y_g} \\
 \frac{1}{-2y_1} & \frac{1}{-2y_2} & \cdots & \frac{1}{-2y_g} \end{pmatrix}.
\end{gather*}
This implies
\begin{gather*}
 \sum_{j=1}^g J_{pj} \frac{x_k^{g-j}}{-2y_k} = \delta_{pk}.
\end{gather*}
Comparing the last equation with \eqref{Pjdef}, we obtain
\begin{gather*}
 J_{pj} = \frac{-2y_p(-1)^{j-1} s_{j-1}^{(p)} }{\prod_{\substack{l\neq p \\ l=1}}^g (x_p-x_l)}.
\end{gather*}
Thus,
\begin{gather*}
 \frac{\partial x_p}{\partial v_n} = \sum_{j=1}^g J_{pj}\omega_{jn} = 
  -2y_p \frac{\sum_{j=1}^g(-1)^{j-1} s_{j-1}^{(p)} \omega_{jn}}{\prod_{\substack{l\neq p \\ l=1}}^g (x_p-x_l)}.
\end{gather*}
\end{proof}

\subsection{Verification}
All formulas and relations given in the paper are verified 
by direct computation of left and right hand sides.
Curves with real branch points are used.
Period matrices $\omega$, $\omega'$, $\eta$, $\eta'$ are computed explicitly, 
as well as matrices $\tau$ and $\varkappa$ for each curve.
Hyperelliptic curves of genera $3$, $4$, $5$, and $6$, 
and theta functions with characteristics up to multiplicity~$3$ are taken.

\section{Main Theorem}
\begin{teo}[General Thomae formula]\label{T:ThN}
 Let $\I_\mFr\cup \J_\mFr$ with $\I_\mFr = \{i_1$, \ldots, 
 $i_{g+1-2\mFr}\}$ and $\J_\mFr = \{j_1,\,\dots,\,j_{g+1+2\mFr}\}$ 
 be a partition of the set $\{1,2,\dots,2g+2\}$ of indices of branch points of hyperelliptic curve, and
 $[\I_\mFr]$ denote the singular characteristic of multiplicity $\mFr$ 
 corresponding to $\mathcal{A} (\I_\mFr) + K$. 
 Then for any $n_1, \dots, n_m \in \{1, ..., g\}$ and any set $\K \subset \J_\mFr$ of cardinality
 $\kFr = 2\mFr-1$ or $2\mFr$ the following relation holds
 \begin{multline}\label{thomaeN}
 \frac{\partial }{\partial v_{n_1}} \cdots 
 \frac{\partial }{\partial v_{n_\mFr}} \theta[\I_\mFr](v)\big|_{v=0} 
 = \epsilon \bigg(\frac{\det \omega}{\pi^g}\bigg)^{1/2} \Delta(\I_\mFr)^{1/4} \Delta(\J_\mFr)^{1/4} \times \\ \times
  \sum_{\substack{p_1,\dots,p_\mFr \in \K \\ \text{all different}}}
 \prod_{i=1}^\mFr \frac{\sum_{j=1}^g  (-1)^{j-1} s_{j-1}(\I_\mFr \cup \K^{(p_i)})\omega_{jn_i}}
 {\prod_{k\in\K \backslash\{p_1,\dots,p_\mFr \}} (e_{p_i} - e_k)}.
\end{multline}
where $\K^{(p_i)}=\K\backslash \{p_i\}$,
$\epsilon$ satisfies $\epsilon^8=1$, then $\Delta(\I_\mFr)$, $\Delta(\J_\mFr)$ denote Vandermonde determinants 
 built from $\{e_i\mid i\in \I_\mFr\}$ and $\{e_j\mid j\in \J_\mFr\}$, 
 and $s_{j}(\I)$ denotes the elementary symmetric polynomial of degree $j$ in $\{e_i\mid i\in \I\}$. 
 The relation does not depend on the choice of $\K$.
\end{teo}

\begin{rem}
 Formula \eqref{thomaeN}, in particular, gives first Thomae formula when $\mFr=0$, 
 and second Thomae formula when $\mFr=1$ and $\kFr = 2\mFr-1$.  That is why it is called general.
 However, the proof given below covers only the cases of $\kFr \geqslant 2$.  
\end{rem}

\begin{proof}[Proof of Theorem~\ref{T:ThN}]
We start with the known result from \cite[\S\,209, p.\;312]{bak897}
with factor  $\sqrt{\Delta(\K)}$ obtained in \cite[Sect.\,VII p.\,348--350]{bak898}.
\newtheorem{ThQteo}{Theta quotient theorem}
\begin{ThQteo}\label{T:thetaquotientMult}
Let $\mathcal{D}= \sum_{k=1}^g P_k$ be a non-special divisor of $g$ finite points $\{P_k=(x_k,y_k)\}$ of a hyperelliptic curve of genus $g$,
and $\K$ be a subset of $\{1,2, \dots, 2g+1\}$ of cardinality $\kFr$, 
and $[\varepsilon_\kappa]$ denote the characteristic of a branch point $e_\kappa$. Then 
\begin{multline}\label{thetaquotientMult}
\frac{\theta\big[\sum_{\kappa\in\K} \varepsilon_{\kappa}\big] \big(v(\mathcal{D}) + K\big)\theta 
\big(v(\mathcal{D}) + K\big)^{\kFr-1}}
{\prod_{\kappa\in \K} \theta [\varepsilon_{\kappa}] \big(v(\mathcal{D}) + K\big)} 
=  \frac{\epsilon \sqrt{\Delta(\K)}}{\Delta[x_1,\dots,x_g]} \times \\ \times
\det \bigg\|\frac{y_k x_k^\lFr}{\phi_{\K}(x_k)},
\frac{y_k x_k^{\lFr-1}}{\phi_{\K}(x_k)},\dots,\frac{y_k}{\phi_{\K}(x_k)},
x_k^\nFr,x_k^{\nFr-1},\dots,1\bigg\|_{k=1}^g,
\end{multline}
where $\epsilon^4=1$, then
$\lFr=[\kFr/2]-1$ and $\nFr=g-1-[\kFr/2]$ for $\kFr\geqslant 2$, 
here $[\cdot]$ denotes the integer part, and
$\phi_{\K}(x)=\prod_{\kappa \in \K} (x-e_\kappa) $.
\end{ThQteo}
As mentioned, formula \eqref{thetaquotientMult} works for $\kFr \geqslant 2$. If $\kFr=1$, 
which is the case of second Thomae formula, one should assign $\nFr=g$, and entries with $y_k$ are absent.
So \eqref{thetaquotientMult} turns into a trivial identity, and Lemma~\ref{L:thetaquotient} is used instead.

\begin{ThQteo}\label{T:thetaquotientK}
Under the assumptions of Theta quotient theorem~\ref{T:thetaquotientMult} the following relation holds
\begin{gather}\label{thetaquotientK}
\frac{\theta[\varepsilon(\K)] \big(v(\mathcal{D}) + K\big)}
{\theta \big(v(\mathcal{D}) + K\big)} 
= \frac{\epsilon}{\big(\prod_{\substack{\kappa \in \K\\ j \in \K^\ast}} (e_\kappa - e_j)\big)^{1/4}}
\frac{\Phi_{\K}[x_1,\dots,x_g] }{\Delta[x_1,\dots,x_g]},
\end{gather}
where $[\varepsilon(\K)]=\sum_{\kappa\in\K} [\varepsilon_{\kappa}]$, 
and $ \K^\ast$ denotes the complement set to $\K$ that is $\K \cup \K^\ast$ serves as 
a  partition of all $2g+1$ indices of finite branch points, 
and the following notation is used
\begin{subequations}
\begin{align}
&\Delta [x_1,\dots,x_g] = \det \hat{V},\qquad \hat{V} = \big\|x_k^{g-1},x_k^{g-2},\dots,x_k,1\big\|_{k=1}^g,\\
&\Phi_{\K}[x_1,\dots,x_g] = \det \hat{\Phi}_{\K},\\
&\hat{\Phi}_{\K} = \bigg\|x_k^\lFr \sqrt{\phi_{\K}^\ast(x_k)},
x_k^{\lFr-1} \sqrt{\phi_{\K}^\ast(x_k)}, \dots, \sqrt{\phi_{\K}^\ast(x_k)}, \label{PhiM} \\
&\qquad\qquad \notag x_k^\nFr \sqrt{\phi_{\K}(x_k)},x_k^{\nFr-1} \sqrt{\phi_{\K}(x_k)},
\dots,\sqrt{\phi_{\K}(x_k)}\bigg\|_{k=1}^g, 
\end{align} 
\end{subequations}
here $\phi_{\K}^\ast(x)= \prod_{j\in \K^\ast} (x-e_j)$.
\end{ThQteo}
\begin{proof} 
Combining Theta quotient theorem~\ref{T:thetaquotientMult} and Lemma~\ref{L:thetaquotient}  we find
\begin{multline*}
\frac{\theta[\varepsilon(\K)] \big(v(\mathcal{D}) + K\big)}
{\theta \big(v(\mathcal{D}) + K\big)} 
=  \frac{\epsilon \sqrt{\Delta(\K)}}{\big(\prod_{\kappa \in \K} \partial_x f(e_\kappa,0)\big)^{1/4}} \times \\ \times
\frac{\sqrt{\prod_{k=1}^g \phi(x_k)}}{\Delta[x_1,\dots,x_g]} 
\det \bigg\|\frac{y_k x_k^\lFr}{\phi_{\K}(x_k)},
\frac{y_k x_k^{\lFr-1}}{\phi_{\K}(x_k)},\dots,\frac{y_k}{\phi_{\K}(x_k)},
x_k^\nFr,x_k^{\nFr-1},\dots,1\bigg\|_{k=1}^g.
\end{multline*}
Then we take into account that 
\begin{gather*}
 \frac{\sqrt{\Delta(\K)}}{\big(\prod_{\kappa \in \K} \partial_x f(e_\kappa,0)\big)^{1/4}} 
 = \frac{1}{\big(\prod_{\substack{\kappa \in \K\\ j \in \K^\ast}} (e_\kappa - e_j)\big)^{1/4}},
\end{gather*}
and
\begin{gather*}
y_k = \bigg(\prod_{l=1}^{2g+1} (x_k-e_l) \bigg)^{1/2} = 
\sqrt{\phi_{\K}(x_k) \phi_{\K}^\ast(x_k)}. 
\end{gather*}
\end{proof}

Derivation of \eqref{thetaquotientK} leads to
 \begin{multline}\label{thetaquotientKD1}
  \frac{\partial}{\partial v_n} \frac{\theta [\varepsilon(\K)] \big(v(\mathcal{D})+ K\big)}
  {\theta \big(v(\mathcal{D})+ K\big)} 
 = \frac{\epsilon}{\big(\prod_{\substack{\kappa \in \K\\ j \in \K^\ast}} (e_i-e_j)\big)^{1/4}}
 \times \\ \times  \sum_{p=1}^g  \frac{\partial x_p}{\partial v_n} 
 \bigg(\frac{\Phi_{\K}^{(p)}}{\Delta} -  \Phi_{\K}\frac{\Delta^{(p)}}{\Delta^2}\bigg),
 \end{multline}
where $\Phi_{\K}$ and $\Delta$ are constructed with the set of points $\{x_1,\dots,x_g\}$ of divisor $\mathcal{D}$, 
and the notation $\Phi_{\K}^{(p)} = \partial \Phi_{\K}/\partial x_p$, and $\Delta^{(p)} = \partial \Delta/\partial x_p$ is adopted.
Since each row of $\hat{\Phi}_{\K}$ depends on a particular point $x_k$, derivative of $\det \hat{\Phi}_{\K}$ with repsect to $x_p$
is the determinant of matrix $\hat{\Phi}_{\K}^{(p)}$ with the same entries as in $\hat{\Phi}_{\K}$ 
except for $p$-th row which is replaced by its derivative. Namely, the $p$-th row has the form
\begin{multline}\label{DPhiRowp}
 \big(\hat{\Phi}_{\K}^{(p)}\big)_{p\cdot}
   = \frac{1}{2y_p}\bigg( x_p^\lFr \phi_{\K}^\ast{}'(x_p) \sqrt{\phi_{\K}(x_p)} 
 + 2 \lFr y_p x_p^{\lFr-1} \sqrt{\phi_{\K}^\ast{}(x_p)}, \\ \dots, 
  \phi_{\K}^\ast{}'(x_p) \sqrt{\phi_{\K}(x_p)},
 x_p^\nFr \phi_{\K}'(x_p) \sqrt{\phi_{\K}^\ast(x_p)} 
 + 2 \nFr y_p x_p^{\nFr-1} \sqrt{\phi_{\K}(x_p)},\\ \dots,
 \phi_{\K}'(x_p) \sqrt{\phi_{\K}^\ast(x_p)} \bigg).
\end{multline}
The same is true for $\hat{V}$.

Let $\I_0$ be a set of $g$ indices of finite branch points, and 
divisor $\mathcal{D}$ consists of branch points $\{(e_i,0)\}_{i\in \I_0}$,
namely: $v(\mathcal{D}) = \mathcal{A}(\I_0)$.
At the same time let $\K\subset \I_0$, that is 
$\{[\varepsilon_{\kappa}] \mid \kappa\in\K\}$ are characteristics of branch points from~$\mathcal{D}$.
We also denote $\I_0 \backslash \K = \I_{\mFr}$.
With these assumptions the left hand side of \eqref{thetaquotientK} equals
\begin{gather*}
\frac{\theta[\varepsilon(\K)] \big(\mathcal{A}(\I_0) + K\big)}
{\theta \big(\mathcal{A}(\I_0) + K\big)} 
= \frac{\theta[\varepsilon(\I_\mFr) + \varepsilon_K ] (0)}
{\theta[ \varepsilon(\I_0) + \varepsilon_K] (0)} 
\end{gather*}
and the numerator vanishes with order $\mFr=[(\kFr+1)/2]$.

In what follows we suppose $\mathcal{D} = \sum_{i\in \I_0} (e_i,0)$, 
and the points are arranged in the order $\K$, $\I_\mFr$. 
Because  $\phi_{\K}(e_\kappa) = 0$ for all $\kappa \in \K$, 
and $\phi_{\K}^\ast (e_\iota) = 0$ for all $\iota \in \I_{\mFr}$ since $\I_\mFr \subset \K^\ast$,
matrix $\hat{\Phi}_{\K}$ consists of four blocks, the right upper block of size 
$\kFr\times (g-[\kFr/2])$ and the left lower block of size $(g-\kFr)\times [\kFr/2]$
are zero. In more detail,
\begin{gather*}
 \begin{array}{cl} & \overbrace{}^{[\kFr/2]} \overbrace{}^{g-[\kFr/2]} \\
 \hat{\Phi}_{\K} =  &
\begin{pmatrix} \mathrm{B}_{\K} & 0 \\ 0 & \mathrm{B}_{\K}^\ast \end{pmatrix} 
\begin{array}{l} \}^\kFr \\  \}^{g-\kFr} \end{array}
 \end{array}
\end{gather*}
where
\begin{align*}
 &\mathrm{B}_{\K} = \begin{pmatrix}
  e_{\kappa_1}^\lFr \sqrt{\phi_{\K}^\ast(e_{\kappa_1})} & 
  e_{\kappa_1}^{\lFr-1} \sqrt{\phi_{\K}^\ast(e_{\kappa_1})} & \dots & \sqrt{\phi_{\K}^\ast(e_{\kappa_1})} \\
  e_{\kappa_2}^\lFr \sqrt{\phi_{\K}^\ast(e_{\kappa_2})} & 
  e_{\kappa_2}^{\lFr-1} \sqrt{\phi_{\K}^\ast(e_{\kappa_2})} & \dots & \sqrt{\phi_{\K}^\ast(e_{\kappa_2})} \\
  \vdots & \vdots & \ddots & \vdots \\
  e_{\kappa_\kFr}^\lFr \sqrt{\phi_{\K}^\ast(e_{\kappa_\kFr})} & 
  e_{\kappa_\kFr}^{\lFr-1} \sqrt{\phi_{\K}^\ast(e_{\kappa_\kFr})} & 
  \dots & \sqrt{\phi_{\K}^\ast(e_{\kappa_\kFr})} 
  \end{pmatrix},\\
 &\mathrm{B}_{\K}^\ast = \begin{pmatrix}
  e_{\iota_1}^{\nFr}\sqrt{\phi_{\K}(e_{\iota_1})} & e_{\iota_1}^{\nFr-1}\sqrt{\phi_{\K}(e_{\iota_1})} 
  & \dots & \sqrt{\phi_{\K}(e_{\iota_1})} \\
  e_{\iota_2}^{\nFr}\sqrt{\phi_{\K}(e_{\iota_2})} & e_{\iota_2}^{\nFr-1}\sqrt{\phi_{\K}(e_{\iota_2})} 
  & \dots & \sqrt{\phi_{\K}(e_{\iota_2})} \\
  \vdots & \vdots & \ddots & \vdots  \\
  e_{\iota_{g-\kFr}}^{\nFr}\sqrt{\phi_{\K}(e_{\iota_{g-\kFr}})} & 
  e_{\iota_{g-\kFr}}^{\nFr-1}\sqrt{\phi_{\K}(e_{\iota_{g-\kFr}})} & 
  \dots & \sqrt{\phi_{\K}(e_{\iota_{g-\kFr}})} 
                   \end{pmatrix}
\end{align*}
with $\K=\{\kappa_1,\kappa_2,\dots,\kappa_{\kFr}\}$, and 
$\I_\mFr = \{\iota_1,\iota_2,\dots, \iota_{g-\kFr}\}$.
Recall that $\nFr + \lFr +2 = g$. Evidently, $\det \hat{\Phi}_{\K} = 0$ if $\kFr \geqslant 1$.

Next, analyse derivative $\partial \det \hat{\Phi}_{\K}/\partial x_p$ which we denote by $\Phi_{\K}^{(p)}$.
Taking into account \eqref{DPhiRowp} we find for $\kappa\in \K$ and $\iota \in \I_\mFr$
\begin{align*}
  &\lim_{x_p \to e_\kappa} y_p (\hat{\Phi}_{\K}^{(p)})_{p\cdot} = \frac{1}{2}\bigg(0, 0, \dots, 0 ,
  e_\kappa^\nFr, e_\kappa^{\nFr-1},\dots, 1 \bigg) 
  \phi_{\K}'(e_\kappa)  \sqrt{\phi_{\K}^\ast(e_\kappa)},\\
  &\lim_{x_p \to e_\iota} y_p (\hat{\Phi}_{\K}^{(p)})_{p\cdot} = \frac{1}{2}\bigg(
  e_\iota^\lFr, e_\iota^{\lFr-1},\dots, 1, 0, 0, \dots, 0 \bigg) 
  \phi_{\K}^\ast{}'(e_\iota)  \sqrt{\phi_{\K}(e_\iota)},
\end{align*}
that is row $(\hat{\Phi}_{\K}^{(\kappa)})_{\kappa\cdot}$ makes a contribution to block $\mathrm{B}_{\K}^\ast$,
and row $(\hat{\Phi}_{\K}^{(\iota)})_{\iota\cdot}$ makes a contribution to block $\mathrm{B}_{\K}$.
The term $y_p$ at branch point $(e_p,0)$ vanishes, however we keep it
as far as it cancels with the same term in $\partial e_p/\partial v_n$.
It is easy to observe that 
\begin{itemize}
 \item $\det \hat{\Phi}_{\K}^{(\iota)} = 0$ for $\iota \in \I_\mFr$ since block matrix $\hat{\Phi}_{\K}^{(\iota)}$ is non-diagonal, 
 \item $\det \hat{\Phi}_{\K}^{(j,j)} = 0$ and $\det \hat{\Phi}_{\K}^{(j,j,\dots)} = 0$ for any index $j$;
 \item $\det \hat{\Phi}_{\K}^{(p_1,\dots,p_m)} = 0$ when all $p_1,\dots,p_m$ are different and $m<\mFr$,
 \item $\det \hat{\Phi}_{\K}^{(p_1,\dots,p_\mFr)} \neq 0$ when all $p_1,\dots,p_\mFr$ are different and the condition holds
 $\kFr = \mFr+[\kFr/2]$, that is $\kFr=2\mFr-1$ or $\kFr=2\mFr$.
\end{itemize}
Thus, $\det \hat{\Phi}_{\K}^{(p_1,\dots,p_\mFr)}$ does not vanish when blocks $\mathrm{B}_{\K}$ and $\mathrm{B}_{\K}^\ast$
are both square. For the sake of simplicity we keep the notation $\mathrm{B}_{\K}$,  $\mathrm{B}_{\K}^\ast$ 
for diagonal non-vanishing blocks of $\hat{\Phi}_{\K}^{(p_1,\dots,p_\mFr)}$. 
At $\kFr=2\mFr-1$, block $\mathrm{B}_{\K}^\ast$ is of size $(g-\mFr+1)$ and $\mathrm{B}_{\K}$ of size $(\mFr-1)$.
At $\kFr=2\mFr$, block $\mathrm{B}_{\K}^\ast$ is of size $(g-\mFr)$ and $\mathrm{B}_{\K}$ has size $\mFr$.
In the both cases
\begin{multline*}
 \det \hat{\Phi}_{\K}^{(p_1,\dots,p_\mFr)} = \Delta[\K\backslash \{p_1,\dots,p_\mFr\}]
 \Delta[\I_{\mFr}\cup \{p_1,\dots,p_\mFr\}]
  \times \\ \times \prod_{\kappa\in \K}  \sqrt{\phi_{\K}^\ast(e_{\kappa})}
 \prod_{\iota \in \I_{\mFr}} \sqrt{\phi_{\K}(e_{\iota})} 
 \prod_{i=1}^\mFr \frac{\phi_{\K}'(e_{p_i})}{2y_{p_i}}.
\end{multline*}

Firstly, we consider the case of $\mFr=2$. Taking the second derivative of \eqref{thetaquotientK}
we find
 \begin{multline*}
  \frac{\partial}{\partial v_{n_1}} \frac{\partial}{\partial v_{n_2}} 
  \frac{\theta [\varepsilon (\K)] \big(v(\mathcal{D}) + K\big)}
{\theta \big(v(\mathcal{D}) + K\big)} 
= \frac{\epsilon}{\big(\prod_{\substack{\kappa\in \K \\ j \in \K^\ast}} (e_\kappa - e_j)\big)^{1/4}} 
 \times \\ \times 
 \sum_{p} \sum_{q} \frac{\partial x_p}{\partial v_{n_1}}  \frac{\partial x_q}{\partial v_{n_2}} 
 \bigg(\frac{\Phi_{\K}^{(p,q)} }{\Delta} - \Phi_{\K} \frac{\Delta^{(p,q)}}{\Delta^2}  \\
 - \Phi_{\K}^{(p)} \frac{\Delta^{(q)}}{\Delta^2} - \Phi_{\K}^{(q)} \frac{\Delta^{(p)}}{\Delta^2}
 + 2 \Phi_{\K}\frac{\Delta^{(p)}\Delta^{(q)}}{\Delta^3} \bigg).
 \end{multline*}
On the right hand side the only non-vanishing term is $\Phi_{\K}^{(p,q)} / \Delta$ with $p\neq q$.
Thus,  assigning all points of divisor $\mathcal{D}$ 
to branch points $\{e_{2l}\}_{l=1}^g$ which form $K = v(\mathcal{D})$, we write down
 \begin{multline*}
  \frac{\partial_{v_{n_1}} \partial_{v_{n_2}} \theta[\varepsilon (\K)] \big(v(\mathcal{D}) + K\big)}
  {\theta \big(v(\mathcal{D}) + K\big)} \bigg|_{v(\mathcal{D})= \mathcal{A}(\I_0)}\\
  = \frac{\epsilon}{\big(\prod_{\substack{\kappa\in \K \\ j \in \K^\ast}} (e_\kappa - e_j)\big)^{1/4}} 
 \sum_{\substack{p\neq q \\ p,q\in \K}} \frac{\partial e_p}{\partial v_{n_1}}  \frac{\partial e_q}{\partial v_{n_2}} 
 \frac{\Phi_{\K}^{(p,q)} }{\Delta[\I_0]}.
 \end{multline*}
 Using Lemma~\ref{L:InvJac} we obtain
\begin{multline*} 
 \text{RHS} =  \epsilon \frac{\big(\prod_{\kappa\in \K } \phi_{\K}^\ast(e_{\kappa}) \big)^{1/4}}
 {\big( \prod_{\iota \in \I_{\mFr}} \phi_{\K}(e_{\iota}) \big)^{1/2}}
 \sum_{\substack{p\neq q \\ p,q\in \K}} \bigg(\sum_{j=1}^g(-1)^{j-1} s_{j-1}^{(p)} \omega_{jn_1} \bigg)
 \times \\ \times   \bigg(\sum_{j=1}^g(-1)^{j-1} s_{j-1}^{(q)} \omega_{jn_2} \bigg) 
  \frac{\Delta[\K\backslash \{p,q\}] \Delta[\{p,q\}]}{\Delta[\K]}.
\end{multline*}
Finally, with $\I_2 = \I_0 \backslash \K$ and $\J_2 = \J_0 \cup \K$, applying first Thomae theorem 
we come to
\begin{multline}\label{thomae3}
 \frac{\partial}{\partial v_{n_1}} \frac{\partial}{\partial v_{n_2}} 
 \theta[\I_2](v) \Big|_{v=0}
 = \epsilon \bigg(\frac{\det \omega}{\pi^g}\bigg)^{1/2} \Delta(\I_2)^{1/4} \Delta(\J_2)^{1/4}  \times \\ \times
 \sum_{\substack{p\neq q \\ p,q\in \K}} \frac{\big(\sum_{j=1}^g(-1)^{j-1} s_{j-1}(\I_0^{(p)}) \omega_{jn_1} \big)
 \big(\sum_{j=1}^g(-1)^{j-1} s_{j-1}(\I_0^{(q)}) \omega_{jn_2} \big)}
 {\prod_{\kappa \in \K\backslash \{p,q\}} (e_p-e_\kappa)(e_q-e_\kappa)},
\end{multline}
where $\I_0^{(k)} = \I_0 \backslash \{k\}$.
The formula works for $\kFr=3$ and $\kFr=4$, where $\kFr$ is the cardinality of $\K$.

With arbitrary $\mFr$, and $\kFr=2\mFr-1$ or $2\mFr$,  we find the $\mFr$-th derivative of \eqref{thetaquotientK},
keeping only the non-vanishing term
 \begin{multline}\label{thetaquotientKDnBP}
  \frac{\partial_{v_{n_1}} \cdots \partial_{v_{n_\mFr}} 
  \theta\big[\sum_{\kappa\in\K} \varepsilon_{\kappa}\big] \big(v(\mathcal{D}) + K\big)}
  {\theta \big(v(\mathcal{D}) + K\big)}\bigg|_{v(\mathcal{D})= K}  \\
  = \frac{\epsilon}{\big(\prod_{\substack{\kappa\in \K \\ j \in \K^\ast}} (e_\kappa - e_j)\big)^{1/4}} 
 \sum_{\substack{p_1,\dots,p_\mFr \in \K \\ \text{all different}}} \frac{\partial e_{p_1}}{\partial v_{n_1}}  \cdots
 \frac{\partial e_{p_\mFr}}{\partial v_{n_\mFr}} \frac{\Phi_{\K}^{(p_1,\dots,p_\mFr)} }{\Delta[\I_0]},
 \end{multline}
where
 \begin{multline*} 
 \text{RHS} =  \frac{\epsilon  \prod_{\kappa\in \K}  \sqrt{\phi_{\K}^\ast(e_{\kappa})}
 \prod_{\iota \in \I_{\mFr}} \sqrt{\phi_{\K}(e_{\iota})}}
 {\big(\prod_{\substack{\kappa\in \K \\ j \in \K^\ast}} (e_\kappa - e_j)\big)^{1/4} \Delta[\I_0]} 
 \times \\ \times 
 \sum_{\substack{p_1,\dots,p_\mFr \in \K \\ \text{all different}}} 
 \frac{\sum_{j=1}^g(-1)^{j-1} s_{j-1}^{(p_1)} \omega_{jn_1}}
  {\prod_{\iota\in \I_\mFr} (e_{p_1}-e_\iota)} \cdots
  \frac{\sum_{j=1}^g(-1)^{j-1} s_{j-1}^{(p_\mFr)} \omega_{jn_\mFr}}
  {\prod_{\iota\in \I_\mFr} (e_{p_\mFr}-e_\iota)}  \times \\ \times  
 \Delta[\K\backslash \{p_1,\dots,p_\mFr\}] \Delta[\I_{\mFr}\cup \{p_1,\dots,p_\mFr\}] \\
   = \epsilon \frac{\big(\prod_{\kappa\in \K } \phi_{\K}^\ast(e_{\kappa}) \big)^{1/4}}
 {\big( \prod_{\iota \in \I_{\mFr}} \phi_{\K}(e_{\iota}) \big)^{1/2}}
 \sum_{\substack{p_1,\dots,p_\mFr \in \K \\ \text{all different}}}
 \frac{\sum_{j=1}^g(-1)^{j-1} s_{j-1}^{(p_1)} \omega_{jn_1}}{\prod_{\kappa\in \K\backslash\{p_1,\dots,p_\mFr\}} (e_{p_1}-e_\kappa)}
 \cdots \times \\ \times  
 \frac{\sum_{j=1}^g(-1)^{j-1} s_{j-1}^{(p_\mFr)} \omega_{jn_\mFr}}{\prod_{\kappa\in \K\backslash\{p_1,\dots,p_\mFr\}}
 (e_{p_\mFr}-e_\kappa)}.
\end{multline*}
Taking into account 
\begin{gather*}
 \prod_{\kappa\in \K } \phi_{\K}^\ast(e_{\kappa})  = \pm
 \prod_{\iota \in \I_{\mFr}} \phi_{\K}(e_{\iota}) \prod_{\substack{\kappa\in \K \\ j \in \J_0}} (e_\kappa - e_j)
\end{gather*}
and applying first Thomae formula we come to \eqref{thomaeN}.
This completes the proof.  
\end{proof}

\section{Corollaries and applications}\label{s:I2Inf}
As the simplest development of general Thomae formula 
we simplify the sums of products of symmetric polynomials in \eqref{thomaeN}.

Firstly, we consider theta functions with non-singular odd characteristics 
obtained by dropping two indices from $\I_0$ (corresponding to
a non-singular even characteristic). Let me emphasize that there exist two types of first derivative theta constants:
whose characteristics are constructed from $g-1$, and from $g-2$ indices. 
The former corresponds to $\I_0$ with one index dropped,
we denote this set by $[\I_1]$. And the latter corresponds to $\I_0$ with two dropped indices,
the corresponding characteristic is denoted by $[\I_1^{\infty}]$, 
here symbol $\infty$ indicates that the index of infinity belongs to this part of partition. 
Similar situation happens with characteristics of arbitrary multiplicity $m$, 
they can be obtained by dropping $2m-1$ or $2m$ indices.
So we have two types of derivative theta constants of each order $m$.

We start with first derivative theta constants corresponding to $[\I_1^{\infty}]$
and find a modification of formulas \eqref{thomae2}--\eqref{thomae2MFu} in this case
(Corollary~\ref{C:thomae2I}). This result seems to be known, since it is inferred 
in some formulas and statements from \cite{EHKKL2011,EHKKLS}.
Then we examine second derivative theta constants, which are naturally arranged in matrices (Hessians).
A matrix representation with simple structure is obtained (Corollary~\ref{C:thomaeK34}),
and examples in genera from $3$ to~$5$ are given.
The representation leads to an essential result (Theorem~\ref{T:DetD2theta})
that the rank of matrices of second derivative theta constants is three in any genus.

Finally, third derivative theta constants are considered.
Simplification of sums of products of symmetric polynomials is made in Corollary~\ref{C:thomaeK56},
and examples of genera $5$ and $6$ are given.  
A conjecture about the simplification for higher derivative theta constants is proposed.

As a byproduct expressions for branch points and their symmetric functions,
which generalize Bolza formulas, are figured out.

\subsection{First derivative theta constants}
Note that second Thomae formula \eqref{thomae2} works only for a partition $\I_1$ 
consisting of $g-1$ indices of finite branch points. 
If $\I_1$ contains $g-2$ indices and the omitted index of infinity, 
here we denote it by $\I_1^{\infty}$, its characteristic $[\I_1^{\infty}]$ is also non-singular odd, 
though \eqref{thomae2} requires a modification.

\begin{cor}[Second Thomae theorem with infinity]\label{C:thomae2I}
 Let $\I_1^{\infty}\cup \J_1$ with $\I_1^{\infty} =\{i_1$, \ldots, $i_{g-2}\}$ and
 $\J_1 = \{j_1$, \ldots, $j_{g+3}\}$ 
 be a partition of the set of $2g+1$ indices of finite branch points, and
 $[\I_1^{\infty}]$ denote the non-singular odd characteristic  
 corresponding to $\mathcal{A} (\I_1^{\infty}) + K$. Then
 \begin{multline}\label{thomae2I}
 \frac{\partial }{\partial v_n} \theta[\I_1^{\infty}](v)\big|_{v=0} \\
 = \epsilon \bigg(\frac{\det \omega}{\pi^g}\bigg)^{1/2} \Delta(\I_1^{\infty})^{1/4} \Delta(\J_1)^{1/4}
\sum_{j=2}^g (-1)^{j-2} s_{j-2}(\I_1^{\infty}) \omega_{jn},
\end{multline}
where  $\epsilon$ satisfies $\epsilon^8=1$, and $\Delta(\I_1^{\infty})$, $\Delta(\J_1)$ 
denote Vandermonde determinants built from $\{e_i\mid i\in \I_1^{\infty}\}$ and $\{e_j\mid j\in \J_1\}$,
then $s_{j}(\I)$ denotes the elementary symmetric polynomial of degree $j$ in $\{e_i\mid i\in \I\}$.
\end{cor}
\begin{proof}
 This is the case of $\kFr=2$ and $\mFr=1$. Let $\I_0=\{i_1,\dots i_g\}$, $\J_0=\{j_1,\dots j_{g+1}\}$, and $\K=\{\kappa_1,\kappa_2\}$. 
Then $\I_1^{\infty} = \I_0 \backslash \K$, and $\J_1 = \J_0 \cup \K$.
Here we start with \eqref{thetaquotientKD1} which reads as
  \begin{multline*}
  \frac{\partial}{\partial v_n} \frac{\theta [\varepsilon(\K)] \big(v(\mathcal{D})- K\big)}
  {\theta \big(v(\mathcal{D})- K\big)} \bigg|_{v(\mathcal{D}) = \mathcal{A}(\I_0)} \\
 = \frac{\epsilon}{\big(\prod_{\substack{\kappa \in \K\\ j \in \K^\ast}} (e_i-e_j)\big)^{1/4}}
  \sum_{p}  \frac{\partial e_p}{\partial v_n} \frac{\Phi_{\K}^{(p)}}{\Delta}.
\end{multline*}
With the help of Lemma~\ref{L:InvJac} and first Thomae theorem one finds
\begin{multline*}
  \text{RHS} =  \epsilon \frac{\big(\prod_{j\in \J_0} (e_{\kappa_1}-e_j)(e_{\kappa_2}-e_j) \big)^{1/4}}
 {\big( \prod_{\iota \in \I_{1}^{(\infty)}} (e_{\iota} - e_{\kappa_1})(e_{\iota} - e_{\kappa_2}) \big)^{1/4}} \times \\ \times
 \bigg(\frac{\sum_{j=1}^g(-1)^{j-1} s_{j-1}\big(\I_0^{(\kappa_1)}\big) \omega_{jn}}{(e_{\kappa_1} - e_{\kappa_2})}
  + \frac{\sum_{j=1}^g(-1)^{j-1} s_{j-1}\big(\I_0^{(\kappa_2)}\big) \omega_{jn}}{(e_{\kappa_2} - e_{\kappa_1})} \bigg) \\
  = \epsilon \frac{\big(\prod_{j\in \J_0} (e_{\kappa_1}-e_j)(e_{\kappa_2}-e_j) \big)^{1/4}}
 {\big( \prod_{\iota \in \I_{1}^{(\infty)}} (e_{\iota} - e_{\kappa_1})(e_{\iota} - e_{\kappa_2}) \big)^{1/4}} 
 \sum_{j=2}^g(-1)^{j-2} s_{j-2}\big(\I_1^{(\infty)}\big) \omega_{jn}.
 \end{multline*}
 where $\I_0^{(\kappa)}=\I_0\backslash \{\kappa\}$.
 This leads to \eqref{thomae2I}.
\end{proof}
In a matrix form \eqref{thomae2I} looks as follows
\begin{multline}\label{thomae2IMF}
 \begin{pmatrix} \partial_{v_1} \\ \partial_{v_2}  \\ \vdots \\ \partial_{v_g}  \end{pmatrix} 
 \theta[\I_1^{\infty}] (v) \big|_{v=0} \\
 = \epsilon \bigg(\frac{\det \omega}{\pi^g}\bigg)^{1/2} \Delta(\I_1^{\infty})^{1/4} \Delta(\J_1)^{1/4}
 \omega^t  \begin{pmatrix} 0 \\ s_0(\I_1^{\infty})  \\ \vdots \\ (-1)^{g-2} s_{g-2}(\I_1^{\infty}) \end{pmatrix}.
\end{multline}
In terms of non-normalized variables $u$ 
\begin{multline}\label{thomae2IMFu}
 \begin{pmatrix} \partial_{u_1} \\ \partial_{u_3}  \\ \vdots \\ \partial_{u_{2g-1}}  \end{pmatrix}  
 \theta[\I_1^{\infty}] (\omega^{-1} u )\big|_{u=0} \\
 = \epsilon \bigg(\frac{\det \omega}{\pi^g}\bigg)^{1/2}
 \Delta(\I_1^{\infty})^{1/4} \Delta(\J_1)^{1/4}
  \begin{pmatrix} 0 \\ s_0(\I_1^{\infty})  \\ \vdots \\ (-1)^{g-2} s_{g-2}(\I_1^{\infty}) \end{pmatrix}.
\end{multline}
This result is used in \cite[Proposition 4.3 p.\,911]{EHKKL2011} to obtain a generalization of Bolza formulas.

\begin{rem}
We introduce a constant related to a genus $g$ curve with a fixed basis of homologies and holomorphic differentials
 \begin{gather}\label{CDef}
 C_g = \bigg(\frac{\det \omega}{\pi^g}\bigg)^{1/2} \Delta^{1/4},
\end{gather}
where $\Delta$ denotes the Vandermonde determinant in all branch points of the curve.
This constant arises in formula \eqref{SigmaDef} connecting theta and sigma functions. 
From first Thomae theorem it follows
\begin{gather*}
 C_g = \epsilon \theta[\I_0] \Big(\prod_{\kappa\in \I_0} \prod_{j \in \J_0} (e_\kappa - e_j)\Big)^{1/4},
\end{gather*}
and $C_g$ is independent of the partition $\I_0\cup \J_0$. With the right ordering we get rid of $\epsilon$,
here $\epsilon=-1$ as follows from computation.
\end{rem}

\begin{exam}\label{E:genus2}
In genus $2$ case formula \eqref{thomae2IMFu} reads as
\begin{gather}\label{omega1G2}
 \begin{pmatrix} \partial_{u_1} \\ \partial_{u_3} \end{pmatrix} 
 \theta[\{\}]  = - C_2 \begin{pmatrix} 0 \\ 1 \end{pmatrix},
\end{gather}
where $C_2$ is defined by \eqref{CDef} with the right ordering of branch points. 
Note that partition $\{\}\cup \{1,2,3,4,5\}$ in genus~$2$ corresponds to 
characteristic $[K]$ of the vector of Riemann constants. Formula \eqref{omega1G2} can be used to determine
constant $C_2$.
\end{exam}
\begin{rem}
This coincides with the result of \cite{Gra1988}, devoted to genus~$2$ case. The fact
that $\partial_{u_3} \theta[\{\}](\omega^{-1} u)$ gives constant $C_2$ 
was proven  with the help of addition theorem by methods of \cite{bak907}. 
The other part of the main result of \cite{Gra1988} connects constant $C_2$ to the product of 
all theta constants (with non-singular even characteristics), which easily extends to higher genera.  
\end{rem}

\begin{teo}\label{T:EvenThetaProd}
For hyperelliptic curve of arbitrary genus $g$ the following relation holds
\begin{gather*}
 \prod_{\text{all }\I_0} \theta[\I_0] = \bigg(\frac{\det \omega}{\pi^g}\bigg)^{\frac{1}{2}\binom{2g+1}{g}} 
 \Delta^{\frac{1}{4}\binom{2g-1}{g} + \frac{1}{4}\binom{2g-1}{g+1}}.
\end{gather*} 
\end{teo}
The proof follows directly from first Thomae formula.

\begin{rem}
In genus $2$ Bolza formulas \cite[Eq.\,(6)]{Bolza}, 
or in more detail in Bolza's dissertaion (G\"{o}ttingen, 1886), see page 15,
\begin{gather*}
 e_i = -\frac{\partial_{u_3} \theta[\{i\}] (\omega^{-1} u)}
 {\partial_{u_1} \theta[\{i\}] (\omega^{-1} u)} \Big|_{u=0} 
\end{gather*}
can be obtained directly from second Thomae formula in the form \eqref{thomae2MFu}, 
which reads as
\begin{gather*}
 \begin{pmatrix} \partial_{u_1} \\ \partial_{u_3} \end{pmatrix} 
 \theta[\{i\}] (\omega^{-1} u) \big|_{u=0} = \frac{-C_2}{\big(\prod_{\substack{j \neq i \\ j=1}}^{2g+1} (e_j-e_i)\big)^{1/4}} 
 \begin{pmatrix} 1 \\ -e_i \end{pmatrix}.
\end{gather*}
\end{rem}

\begin{exam}\label{E:BolzaG3}
In genus $3$ generalization of Bolza formulas is obtained from \eqref{thomae2IMFu}.
Namely, cf. \cite[Eq.\,(6.26)]{EHKKL2011} and \cite[Eq.\,(6.24)]{EHKKLS}
\begin{gather*}
 e_i = - \frac{\partial_{u_5} \theta[\{i\}] (\omega^{-1} u)}
 {\partial_{u_3} \theta[\{i\}] (\omega^{-1} u)} \Big|_{u=0}.
\end{gather*}
And second Thomae formula \eqref{thomae2MFu}
implies the following, cf. \cite[Eq.\,(6.25)]{EHKKL2011},  \cite[Eq.\,(6.23)]{EHKKLS},
\begin{align*}
 e_i + e_k &= - \frac{\partial_{u_3} \theta[\{i,k\}] (\omega^{-1} u)}
 {\partial_{u_1} \theta[\{i,k\}] (\omega^{-1} u)} \Big|_{u=0},\\
 e_i e_k &= \frac{\partial_{u_5} \theta[\{i,k\}] (\omega^{-1} u)}
 {\partial_{u_1} \theta[\{i,k\}] (\omega^{-1} u)} \Big|_{u=0}.
\end{align*}
\end{exam}
\begin{exam}\label{E:BolzaG4P2}
In genus $4$ formulas for symmetric polynomials in two branch points
follow from \eqref{thomae2IMFu}
\begin{align*}
 e_i + e_k &= - \frac{\partial_{u_5} \theta[\{i,k\}] (\omega^{-1} u)}
 {\partial_{u_3} \theta[\{i,k\}] (\omega^{-1} u)} \Big|_{u=0},\\
 e_i e_k &= \frac{\partial_{u_7} \theta[\{i,k\}] (\omega^{-1} u)}
 {\partial_{u_3} \theta[\{i,k\}] (\omega^{-1} u)} \Big|_{u=0}.
\end{align*}
\end{exam}
In \cite[Proposition 4.3 p.\,911]{EHKKL2011} a generalization of Bolza formulas coming from \eqref{thomae2MFu} and 
\eqref{thomae2IMFu} was found in terms of sigma function at half-periods.
In our notation this generalization has the form
\begin{align}
 &s_{j}(\I_1) = (-1)^j \frac{\partial_{u_{2j+1}} \theta[\I_1] (\omega^{-1} u)}
 {\partial_{u_1} \theta[\I_1] (\omega^{-1} u)} \Big|_{u=0},\quad j=1,\dots,g-1;\\
 &s_{j}(\I_1^{\infty}) = (-1)^j \frac{\partial_{u_{2j+3}} \theta[\I_1] (\omega^{-1} u)}
 {\partial_{u_3} \theta[\I_1] (\omega^{-1} u)} \Big|_{u=0},\quad j=1,\dots,g-2.
\end{align}

\begin{rem}
 With the help of first derivative theta constants one can find expressions 
 for symmetric polynomials in $g-1$ or $g-2$ branch points, these correspond to partitions $\I_1$ and $\I_1^{\infty}$.
 Expressions for separate branch points arise for genera $2$ and $3$ only. 
 For higher genera one should use higher derivative theta constants.
\end{rem}

\subsection{Second derivative theta constants}
Next we consider characteristics of multiplicity $2$, which arise when $3$ or $4$ indices drop from~$\I_0$.
Again $\I_0 \cup \J_0$ is a partition of $2g+1$ indices of finite branch points with 
$\I_0=\{i_1,\dots i_g\}$, and $\J_0=\{j_1,\dots j_{g+1}\}$. Let
$\K$ denote the set of indices which drop,
then $\I_2 = \I_0 \backslash \K$, and $\J_2 = \J_0 \cup \K$.

\begin{cor}\label{C:thomaeK34}
 Let $\I_2\cup \J_2$ with $\I_2=\{i_1$, \ldots, $i_{g-\kFr}\}$ and $\J_2 = \{j_1$, \ldots, $j_{g+1+\kFr}\}$, 
 where $\kFr=3$ or $4$,
 be a partition of the set of $2g+1$ indices of finite branch points, 
 such that singular characteristic $[\I_2]$, 
 corresponding to $\mathcal{A} (\I_2) + K$,
 has multiplicity $2$. 
 Let $\Delta(\I_2)$ and $\Delta(\J_2)$ be Vandermonde determinants 
 built from $\{e_i\mid i\in \I_2\}$ and $\{e_j\mid j\in \J_2\}$. Then
 \begin{multline}\label{thomaeK3}
 \frac{\partial}{\partial v_{n_1}} \frac{\partial}{\partial v_{n_2}} 
 \theta[\I_2](v) \big|_{v=0} = \epsilon \bigg(\frac{\det \omega}{\pi^g}\bigg)^{1/2} \Delta(\I_2)^{1/4} \Delta(\J_2)^{1/4} 
 \times \\ \times  
 \sum_{i,j=1}^g (-1)^{i+j} \Big(2s_{i-\kFr+1}(\I_2) s_{j-\kFr+1}(\I_2) \\
 - s_{i-\kFr+2}(\I_2) s_{j-\kFr}(\I_2)
 - s_{i-\kFr}(\I_2) s_{j-\kFr+2}(\I_2)\Big) \omega_{i n_1} \omega_{j n_2}.
\end{multline}
where  $\epsilon$ satisfies $\epsilon^8=1$, and elementary symmetric functions 
$s_l(\I_2)$  are replaced by zero when $l<0$.
\end{cor}
In matrix form 
 \begin{gather}\label{thomae3MF}
 \partial_v^2  \theta[\I_2](v) \big|_{v=0} 
 = \epsilon \bigg(\frac{\det \omega}{\pi^g}\bigg)^{1/2} \Delta(\I_2)^{1/4} \Delta(\J_2)^{1/4} 
 \omega^t \hat{S}[\I_2] \omega,
\end{gather}
where $\partial_v^2$ denotes the operator of second derivatives, whose 
entries are $\partial_{v_{n_i}} \partial_{v_{n_j}}$,
and $\hat{S}[\I_2]$ is a $g \times g$ matrix with entries
\begin{multline}\label{SmatrDef}
 (\hat{S}[\I_2])_{i,j} = (-1)^{i+j} \Big(2s_{i-\kFr+1}(\I_2) s_{j-\kFr+1}(\I_2) \\
 - s_{i-\kFr+2}(\I_2) s_{j-\kFr}(\I_2)
 - s_{i-\kFr}(\I_2) s_{j-\kFr+2}(\I_2)\Big).
\end{multline}
With non-normalized variables $u$
 \begin{gather}\label{thomae3MFu}
 \partial_u^2  \theta[\I_2](\omega^{-1} u) \big|_{u=0} 
 = \epsilon \bigg(\frac{\det \omega}{\pi^g}\bigg)^{1/2} \Delta(\I_2)^{1/4} \Delta(\J_2)^{1/4} \hat{S}[\I_2].
\end{gather}

\begin{proof}[Proof of Corollary~\ref{C:thomaeK34}]
Recall that multiplicity $\mFr=2$ arises when $\kFr=3$ or $4$. 
First consider the case of three dropped indices with $\K=\{\kappa_1$, $\kappa_2$, $\kappa_3\}$. 
Starting from \eqref{thomae3}, by straightforward computation we find
\begin{multline*}
 \sum_{\substack{p\neq q \\ p,q\in \K}} \frac{\big(\sum_{j=1}^g(-1)^{j-1} s_{j-1}(\I_0^{(p)}) \omega_{j n_1} \big)
 \big(\sum_{j=1}^g(-1)^{j-1} s_{j-1}(\I_0^{(q)}) \omega_{jn_2} \big)}
 {\prod_{\kappa \in \K\backslash \{p,q\}} (e_p-e_\kappa)(e_q-e_\kappa)} \\
  = \sum_{i,j=1}^g (-1)^{i+j} \Big(2s_{i-2}(\I_2) s_{j-2}(\I_2)\\ - s_{i-1}(\I_2) s_{j-3}(\I_2)
 - s_{i-3}(\I_2) s_{j-1}(\I_2)\Big) \omega_{i n_1} \omega_{j n_2},
\end{multline*}
where $s_{l}(\I_2)$ is replaced by zero  when $l<0$.

In the case of four dropped indices with $\K=\{\kappa_1,\kappa_2, \kappa_3, \kappa_4\}$ similar computation 
leads to the following
\begin{multline*}
 \sum_{\substack{p\neq q \\ p,q\in \K}} \frac{\big(\sum_{j=1}^g(-1)^{j-1} s_{j-1}(\I_0^{(p)}) \omega_{jn_1} \big)
 \big(\sum_{j=1}^g(-1)^{j-1} s_{j-1}(\I_0^{(q)}) \omega_{jn_2} \big)}
 {\prod_{\kappa \in \K\backslash \{p,q\}} (e_p-e_\kappa)(e_q-e_\kappa)} \\
  = \sum_{i,j= 2}^g (-1)^{i+j} \Big(2s_{i-3}(\I_2) s_{j-3}(\I_2)\\ - s_{i-2}(\I_2) s_{j-4}(\I_2)
 - s_{i-4}(\I_2) s_{j-2}(\I_2)\Big) \omega_{i n_1} \omega_{j n_2}.
\end{multline*}
These two expressions can be written as \eqref{thomaeK3}.
\end{proof}

\begin{exam}
In genus $3$ case $\I_2 = \{\}$, and
\begin{gather}\label{HessI2g3}
 \hat{S}[\{\}] = \begin{pmatrix} 0 & 0 & -1 \\ 0 & 2 & 0 \\ -1 & 0 & 0 \end{pmatrix}.
\end{gather}
Let $\I_0 = \{i_1,i_2,i_3\}$, then $\J_0$ is the complement to $\I_0$ in the set of $7$ indices of finite branch points,
$\K$ coincides with $\I_0$. So \eqref{thomae3MF} gives
\begin{gather}\label{SconstG3}
 \partial_v^2  \theta[\{\}](v) \big|_{v=0} = 
 - C_3 \omega^t \hat{S}[\{\}] \omega.
\end{gather}
where $C_3$ is defined by \eqref{CDef}.
Representation \eqref{thomae3MFu} in terms of non-normalized variables $u$
allows to determine constant $C_3$, namely
\begin{gather*}
 C_3 = \partial_{u_1,u_5}^2  \theta[\{\}](\omega^{-1} u) \big|_{u=0} = 
 -\frac{1}{2} \partial_{u_3,u_3}^2  \theta[\{\}](\omega^{-1} u) \big|_{u=0}.
\end{gather*}
\end{exam}

\begin{exam}
In genus $4$ there is the unique partition $\I_2^{(\infty)} = \{\}$ with
\begin{gather*}
 \hat{S}[\{\}] = \begin{pmatrix} 0 & 0 & 0 & 0 \\
 0 & 0 & 0 & -1 \\ 0 & 0 & 2 & 0 \\ 
 0 & -1 & 0 & 0  \end{pmatrix}.
\end{gather*}
The relation similar to \eqref{SconstG3} holds
with constant $C_4$, defined by \eqref{CDef}.
Again constant $C_4$ can be found from the equalities
\begin{gather*}
 C_4 = \partial_{u_3,u_7}^2  \theta[\{\}](\omega^{-1} u) \big|_{u=0} = 
 - \frac{1}{2} \partial_{u_5,u_5}^2  \theta[\{\}](\omega^{-1} u) \big|_{u=0}.
\end{gather*}
\end{exam}

\begin{exam}\label{E:BolzaG4}
In genus $4$ there are $2g+1=9$ singular even characteristics $\I_2 = \{\iota\}$ with
\begin{gather*}
 \hat{S}[\{\iota\}] = \begin{pmatrix} 0 & 0 & -1 & e_\iota \\ 0 & 2 & -e_\iota & -e_\iota^2 \\ 
 -1 & -e_\iota & 2e_\iota^2 & 0 \\ e_\iota & - e_\iota^2 & 0 & 0 \end{pmatrix}.
\end{gather*}
Let $\I_0 = \{i_1,i_2,i_3,\iota\}$, and $\J_0$ be the complement to $\I_0$ in the set of $9$ indices of finite branch points,
here $\K=\{i_1,i_2,i_3\}$. Then \eqref{thomae3MFu} gives (with right ordering)
\begin{gather*}
 \partial_u^2 \theta[\{\iota\}](\omega^{-1} u) \big|_{u=0} 
 = \frac{-C_4}{\big(\prod_{\substack{j \neq \iota \\ j=1}}^{2g+1} (e_\iota - e_j)\big)^{1/4}}
  \hat{S}[\{\iota\}] .
\end{gather*}
This immediately implies a generalization of Bolza formulas to genus~$4$
\begin{align}\label{BolzaG4}
  e_\iota &= - \frac{\partial_{u_1,u_7}^2 \theta[\{\iota\}](\omega^{-1} u)}
  {\partial_{u_1,u_5}^2 \theta[\{\iota\}](\omega^{-1} u)} \Big|_{u=0}
   = \frac{\partial_{u_3,u_5}^2 \theta[\{\iota\}](\omega^{-1} u)}
  {\partial_{u_1,u_5}^2 \theta[\{\iota\}](\omega^{-1} u)} \Big|_{u=0} \notag \\
  &= - 2 \frac{\partial_{u_3,u_5}^2 \theta[\{\iota\}](\omega^{-1} u)}
  {\partial_{u_3,u_3}^2 \theta[\{\iota\}](\omega^{-1} u)} \Big|_{u=0} 
  = 2 \frac{\partial_{u_1,u_7}^2 \theta[\{\iota\}](\omega^{-1} u)}
  {\partial_{u_3,u_3}^2 \theta[\{\iota\}](\omega^{-1} u)} \Big|_{u=0}  \\
  & = -\frac{\partial_{u_3,u_7}^2 \theta[\{\iota\}](\omega^{-1} u)}
  {\partial_{u_1,u_7}^2 \theta[\{\iota\}](\omega^{-1} u)} \Big|_{u=0}
  = \frac{\partial_{u_3,u_7}^2 \theta[\{\iota\}](\omega^{-1} u)}
  {\partial_{u_3,u_5}^2 \theta[\{\iota\}](\omega^{-1} u)} \Big|_{u=0} \notag \\
  & = - \frac{1}{2} \frac{\partial_{u_5,u_5}^2 \theta[\{\iota\}](\omega^{-1} u)}
  {\partial_{u_3,u_5}^2 \theta[\{\iota\}](\omega^{-1} u)} \Big|_{u=0}
  = \frac{1}{2} \frac{\partial_{u_5,u_5}^2 \theta[\{\iota\}](\omega^{-1} u)}
  {\partial_{u_1,u_7}^2 \theta[\{\iota\}](\omega^{-1} u)} \Big|_{u=0}. \notag
\end{align}
See Example~\ref{E:BolzaG3} for genus $3$ case, and Example~\ref{E:BolzaG4P2}
for symmetric functions in two branch points in genus $4$.
\end{exam}
\begin{rem}
Another generalization of Bolza formulas, in terms of sigma function,
the reader could find in \cite[Proposition 4.2 p.\,911]{EHKKL2011}.
In the present paper the generalization involves lower derivatives, actually the 
lowest non-vanishing derivatives of theta function with characteristic $[\{\iota\}]$ at $v=0$.
\end{rem}

\begin{exam}\label{E:BolzaG5}
In genus $5$ there exist $2g+1=11$ singular even characteristics of 
multiplicity $2$ corresponding to partitions $\I_2^{\infty} = \{\iota\}$ with matrices
\begin{gather*}
\hat{S}[\{\iota\}] = \begin{pmatrix} 
0 & 0 & 0 & 0 & 0 \\
0 & 0 & 0 & -1 & e_\iota \\
0 & 0 & 2 & -e_\iota & - e_\iota^2 \\
0 & -1 & -e_\iota & 2 e_\iota^2 & 0 \\
0 & e_\iota & -e_\iota^2 & 0 & 0
          \end{pmatrix},
\end{gather*}
and $\binom{2g+1}{2}=55$ characteristics of multiplicity $2$ 
corresponding to partitions $\I_2 = \{\iota_1,\iota_2\}$ with
\begin{multline*}
 \hat{S}[\{\iota_1,\iota_2\}] = \left( \begin{matrix} 
 0 & 0 & -1 \\ 
 0 & 2 & - e_{\iota_1} - e_{\iota_2}  \\ 
 -1 & -e_{\iota_1} - e_{\iota_2} & 2(e_{\iota_1}^2 + e_{\iota_1} e_{\iota_2} + e_{\iota_2}^2) \\ 
 e_{\iota_1} + e_{\iota_2} & - e_{\iota_1}^2 - e_{\iota_2}^2 & - e_{\iota_1} e_{\iota_2} (e_{\iota_1} + e_{\iota_2}) \\ 
 - e_{\iota_1} e_{\iota_2}  & e_{\iota_1} e_{\iota_2} (e_{\iota_1} + e_{\iota_2}) 
 & - e_{\iota_1}^2 e_{\iota_2}^2 
 \end{matrix} \right.  \\  \left.
 \begin{matrix} e_{\iota_1} + e_{\iota_2} & - e_{\iota_1} e_{\iota_2} \\
  - e_{\iota_1}^2 - e_{\iota_2}^2 & e_{\iota_1} e_{\iota_2} (e_{\iota_1} + e_{\iota_2}) \\
  - e_{\iota_1} e_{\iota_2} (e_{\iota_1} + e_{\iota_2}) & -e_{\iota_1}^2 e_{\iota_2}^2 \\
  2 e_{\iota_1}^2 e_{\iota_2}^2 & 0 \\
  0 & 0
\end{matrix} \right).
\end{multline*}
Thus, a generalization of Bolza formulas to genus~$5$ can be obtained
\begin{align}\label{BolzaG5}
  e_\iota &= - \frac{\partial_{u_3,u_9}^2 \theta[\{\iota\}](\omega^{-1} u)}
  {\partial_{u_3,u_7}^2 \theta[\{\iota\}](\omega^{-1} u)} \Big|_{u=0}
   = \frac{\partial_{u_5,u_7}^2 \theta[\{\iota\}](\omega^{-1} u)}
  {\partial_{u_3,u_7}^2 \theta[\{\iota\}](\omega^{-1} u)} \Big|_{u=0} \notag \\
  &= - 2 \frac{\partial_{u_5,u_7}^2 \theta[\{\iota\}](\omega^{-1} u)}
  {\partial_{u_5,u_5}^2 \theta[\{\iota\}](\omega^{-1} u)} \Big|_{u=0} 
  = 2 \frac{\partial_{u_3,u_9}^2 \theta[\{\iota\}](\omega^{-1} u)}
  {\partial_{u_5,u_5}^2 \theta[\{\iota\}](\omega^{-1} u)} \Big|_{u=0} \\
  &= -\frac{\partial_{u_5,u_9}^2 \theta[\{\iota\}](\omega^{-1} u)}
  {\partial_{u_3,u_9}^2 \theta[\{\iota\}](\omega^{-1} u)} \Big|_{u=0}
   = \frac{\partial_{u_5,u_9}^2 \theta[\{\iota\}](\omega^{-1} u)}
  {\partial_{u_5,u_7}^2 \theta[\{\iota\}](\omega^{-1} u)} \Big|_{u=0} \notag \\
  & = - \frac{1}{2} \frac{\partial_{u_7,u_7}^2 \theta[\{\iota\}](\omega^{-1} u)}
  {\partial_{u_5,u_7}^2 \theta[\{\iota\}](\omega^{-1} u)} \Big|_{u=0}
  = \frac{1}{2} \frac{\partial_{u_7,u_7}^2 \theta[\{\iota\}](\omega^{-1} u)}
  {\partial_{u_3,u_9}^2 \theta[\{\iota\}](\omega^{-1} u)} \Big|_{u=0}, \notag
\end{align}
and formulas for symmetric functions in two branch points
\begin{align}\label{Bolza21G5}
  e_{\iota_1} + e_{\iota_2} &= - \frac{\partial_{u_1,u_7}^2 \theta[\{\iota\}](\omega^{-1} u)}
  {\partial_{u_1,u_5}^2 \theta[\{\iota\}](\omega^{-1} u)} \Big|_{u=0}
   = \frac{2\partial_{u_1,u_7}^2 \theta[\{\iota\}](\omega^{-1} u)}
  {\partial_{u_3,u_3}^2 \theta[\{\iota\}](\omega^{-1} u)} \Big|_{u=0} \notag \\
  &= \frac{\partial_{u_3,u_5}^2 \theta[\{\iota\}](\omega^{-1} u)}
  {\partial_{u_1,u_5}^2 \theta[\{\iota\}](\omega^{-1} u)} \Big|_{u=0}
  = -\frac{2\partial_{u_3,u_5}^2 \theta[\{\iota\}](\omega^{-1} u)}
  {\partial_{u_3,u_3}^2 \theta[\{\iota\}](\omega^{-1} u)} \Big|_{u=0} \\
  & = \frac{(\partial_{u_5,u_5}^2+\partial_{u_3,u_7}^2)\theta[\{\iota\}](\omega^{-1} u)}
  {\partial_{u_1,u_7}^2 \theta[\{\iota\}](\omega^{-1} u)} \Big|_{u=0} \notag  \\
  & = -\frac{(\partial_{u_5,u_5}^2+\partial_{u_3,u_7}^2)\theta[\{\iota\}](\omega^{-1} u)}
  {\partial_{u_3,u_5}^2 \theta[\{\iota\}](\omega^{-1} u)} \Big|_{u=0} \notag\\
  &= - \frac{\partial_{u_3,u_9}^2 \theta[\{\iota\}](\omega^{-1} u)}
  {\partial_{u_1,u_9}^2 \theta[\{\iota\}](\omega^{-1} u)} \Big|_{u=0} 
  = \frac{\partial_{u_5,u_7}^2 \theta[\{\iota\}](\omega^{-1} u)}
  {\partial_{u_1,u_9}^2 \theta[\{\iota\}](\omega^{-1} u)} \Big|_{u=0}, \notag
  \\
  e_{\iota_1} e_{\iota_2} &= \frac{\partial_{u_1,u_9}^2 \theta[\{\iota\}](\omega^{-1} u)}
  {\partial_{u_1,u_5}^2 \theta[\{\iota\}](\omega^{-1} u)} \Big|_{u=0} 
  = -\frac{2\partial_{u_1,u_9}^2 \theta[\{\iota\}](\omega^{-1} u)}
  {\partial_{u_3,u_3}^2 \theta[\{\iota\}](\omega^{-1} u)} \Big|_{u=0} \notag \\
  &= -\frac{(\partial_{u_5,u_5}^2+2\partial_{u_3,u_7}^2)\theta[\{\iota\}](\omega^{-1} u)}
  {2\partial_{u_1,u_5}^2 \theta[\{\iota\}](\omega^{-1} u)} \Big|_{u=0} \label{Bolza22G5} \\
  &= \frac{(\partial_{u_5,u_5}^2+2\partial_{u_3,u_7}^2)\theta[\{\iota\}](\omega^{-1} u)}
  {\partial_{u_3,u_3}^2 \theta[\{\iota\}](\omega^{-1} u)} \Big|_{u=0} \notag \\
  & = \frac{\partial_{u_3,u_9}^2 \theta[\{\iota\}](\omega^{-1} u)}
  {\partial_{u_1,u_7}^2 \theta[\{\iota\}](\omega^{-1} u)} \Big|_{u=0} 
  = - \frac{\partial_{u_3,u_9}^2 \theta[\{\iota\}](\omega^{-1} u)}
  {\partial_{u_3,u_5}^2 \theta[\{\iota\}](\omega^{-1} u)} \Big|_{u=0} \notag \\
  &= -\frac{\partial_{u_5,u_7}^2 \theta[\{\iota\}](\omega^{-1} u)}
  {\partial_{u_1,u_7}^2 \theta[\{\iota\}](\omega^{-1} u)} \Big|_{u=0}
   = \frac{\partial_{u_5,u_7}^2 \theta[\{\iota\}](\omega^{-1} u)}
  {\partial_{u_3,u_5}^2 \theta[\{\iota\}](\omega^{-1} u)} \Big|_{u=0} \notag \\
  & = \frac{\partial_{u_5,u_9}^2 \theta[\{\iota\}](\omega^{-1} u)}
  {\partial_{u_1,u_9}^2 \theta[\{\iota\}](\omega^{-1} u)} \Big|_{u=0} 
   = -\frac{\partial_{u_7,u_7}^2 \theta[\{\iota\}](\omega^{-1} u)}
  {2\partial_{u_1,u_9}^2 \theta[\{\iota\}](\omega^{-1} u)} \Big|_{u=0} \notag \\
  &= - \frac{2\partial_{u_5,u_9}^2 \theta[\{\iota\}](\omega^{-1} u)}
  {(\partial_{u_5,u_5}^2 + 2 \partial_{u_3,u_7}^2) \theta[\{\iota\}](\omega^{-1} u)} \Big|_{u=0} \notag \\
  &= \frac{\partial_{u_7,u_7}^2 \theta[\{\iota\}](\omega^{-1} u)}
  {(\partial_{u_5,u_5}^2 + 2 \partial_{u_3,u_7}^2) \theta[\{\iota\}](\omega^{-1} u)} \Big|_{u=0}. \notag
\end{align}
\end{exam}

In genus $g$ with the help of second derivative theta constants one can compute 
symmetric polynomials in $g-4$ and $g-3$ branch points, these are the possible cardinalities of $\I_2$,
and we use the notation $\kFr=3$ or~$4$. The first equalities from \eqref{BolzaG4}, \eqref{BolzaG5}, 
\eqref{Bolza21G5}, and \eqref{Bolza22G5} are combined as
\begin{gather}\label{BolzaI2}
 s_j(\I_2) = (-1)^j \frac{\partial^2_{u_{2\kFr-5},u_{2\kFr+2j-1}} \theta[\I_2](\omega^{-1} u)}
 {\partial^2_{u_{2\kFr-5},u_{2\kFr-1}} \theta[\I_2](\omega^{-1} u)}\Big|_{u=0},
\end{gather}
which holds in arbitrary genus $g$.

\begin{rem}
 Note that matrix $ \hat{S}[\I_2]$ is non-degenerate only in genus $3$.
 In higher genera $ \hat{S}[\I_2]$ has rank $3$.
\end{rem}

\begin{teo}\label{T:DetD2theta}
For hyperelliptic curves of genera $g\geqslant 3$, when characteristics of multiplicity $2$ exist, 
rank of every matrix of second derivative theta constants equals three, that is 
\begin{gather*}
 \rank \big(\partial^2_v \theta[\I_2] \big) = 3.
\end{gather*}
Therefore,  $\det\big(\partial^2_v \theta[\I_2]\big)=0$ in genera $g>3$.
\end{teo}
\begin{proof}
Corollary~\ref{C:thomaeK34} gives a decomposition 
of Hessian matrix $\partial^2_v \theta[\I_2]$ as the product of non-degenerate matrices $\omega$ 
and $g\times g$ matrix $\hat{S}[\I_2]$.
The latter, as easily seen from \eqref{SmatrDef}, is composed of columns spanned by three vectors.
Therefore, rank of $\partial^2_v \theta[\I_2]$ equals $3$.
\end{proof}

It could be observed from \eqref{SmatrDef} that matrix $\hat{S}(\I_2)$ belongs to the second tensor power 
$\mathcal{S}_3^{\otimes 2}$ of 
the vector space $\mathcal{S}_3$ spanned by three vectors $\mathbf{s}_{0}$, $\mathbf{s}_{1}$, $\mathbf{s}_{2}$
such that $\mathbf{s}_{d}=\big(s_{j-\kFr + d}(\I_2)\big)_{j=1}^g$,
recall that $\kFr=3$ or $4$. And tensor rank of $\hat{S}(\I_2)$ is three,
since it is spanned of three basis elements: $\mathbf{s}_{1}\otimes \mathbf{s}_{1}$,
 $\mathbf{s}_{0}\otimes \mathbf{s}_{2}$, $\mathbf{s}_{2}\otimes \mathbf{s}_{0}$.

\subsection{Third derivative theta constants}
Here we consider characteristics of multiplicity $3$, obtained by dropping $5$ or $6$ indices from~$\I_0$.
Again $\K$ denotes the set of dropped indices, $\I_3 = \I_0 \backslash \K$, and $\J_3 = \J_0 \cup \K$.
\begin{cor}\label{C:thomaeK56}
 Let $\I_3\cup \J_3$ with $\I_3=\{i_1$, \ldots, $i_{g-\kFr}\}$ and $\J_3 = \{j_1$, \ldots, $j_{g+1+\kFr}\}$, 
 where $\kFr=5$ or $6$,
 be a partition of the set of indices of $2g+1$ finite branch points, 
 such that singular characteristic $[\I_3]$, 
 corresponding to $\mathcal{A} (\I_3) + K$, has multiplicity $3$. 
 Let $\Delta(\I_3)$ and $\Delta(\J_3)$ be Vandermonde determinants 
 built from $\{e_i\mid i\in \I_3\}$ and $\{e_j\mid j\in \J_3\}$. Then
 \begin{multline}\label{thomaeK4}
 \frac{\partial}{\partial v_{n_1}} \frac{\partial}{\partial v_{n_2}} \frac{\partial}{\partial v_{n_3}} 
 \theta[\I_3](v) \big|_{v=0} 
 = \epsilon \bigg(\frac{\det \omega}{\pi^g}\bigg)^{1/2} \Delta(\I_3)^{1/4} \Delta(\J_3)^{1/4} \times \\ \times
 \sum_{j_1,j_2,j_3 = 1}^g (\hat{S}[\{\I_3\}])_{j_1,j_2,j_3} \omega_{j_1 n_1} \omega_{j_2 n_2} \omega_{j_3 n_3},
\end{multline}
with tensor of  order $3$
\begin{multline}\label{SmatrDef3}
 (\hat{S}[\{\I_3\}])_{j_1,j_2,j_3}  = (-1)^{j_1+j_2+j_3-3\kFr} \Big(6 s_{j_1-\kFr+2} s_{j_2-\kFr+2} s_{j_3-\kFr+2} \\
 - 2 \{ s_{j_1-\kFr+3} s_{j_2-\kFr+2} s_{j_3-\kFr+1} \} + 2 \{s_{j_1-\kFr} s_{j_2-\kFr+3} s_{j_3-\kFr+3} \} \\
 + 2 \{ s_{j_1-\kFr+1} s_{j_2-\kFr+1} s_{j_3-\kFr+4} \} - \{ s_{j_1-\kFr} s_{j_2-\kFr+2} s_{j_3-\kFr+4} \} \Big).
\end{multline}
where $\{ s_{j_1-\kFr+d_1} s_{j_2-\kFr + d_2} s_{j_3-\kFr + d_3}\}$ denotes the sum over 
all permutations of $\{d_1,d_2,d_3\}$, elementary symmetric functions 
$s_l$ are constructed in $\{e_i\mid i\in \I_3\}$, and replaced by zero when $l<0$, and $\epsilon$ satisfies $\epsilon^8=1$.
\end{cor}
In a tensor form with non-normalized variables $u$ \eqref{thomaeK4} reads as
\begin{gather}\label{thomaeK4MFu}
  \partial_u^3  \theta[\I_3](\omega^{-1} u) \big|_{u=0} 
 = \epsilon \bigg(\frac{\det \omega}{\pi^g}\bigg)^{1/2} \Delta(\I_3)^{1/4} \Delta(\J_3)^{1/4} \hat{S}[\I_3].
\end{gather}

To extend Corollary~\ref{C:thomaeK56} to higher multiplicities the following Conjecture could be helpful.
\begin{conj}\label{C:thomaeKH}
Let $\I_\mFr\cup \J_\mFr$ be a partition of the set of indices of $2g+1$ finite branch points 
with $\I_\mFr=\{i_1$, \ldots, $i_{g-\kFr}\}$ and $\J_\mFr = \{j_1$, \ldots, $j_{g+1+\kFr}\}$, 
 where $\kFr=2\mFr-1$ or $2\mFr$,  
 and characteristic $[\I_\mFr]$ of  multiplicity~$\mFr$
 corresponds to $\mathcal{A} (\I_\mFr) + K$. 
In the expansion of
\begin{gather*}
 (\hat{S}[\I_\mFr])_{j_1,\dots,j_\mFr} = \sum_{\substack{p_1,\dots,p_\mFr \in \K \\ \text{all different}}}
 \prod_{i=1}^\mFr \frac{\sum_{j=1}^g  (-1)^{j-1} s_{j-1}(\I_\mFr \cup \K^{(p_i)})}
 {\prod_{k\in\K \backslash\{p_1,\dots,p_\mFr \}} (e_{p_i} - e_k)}
\end{gather*}
over $\prod_{\iota = 1}^\mFr s_{j_\iota - \kFr + d_\iota} (\I_\mFr)$ 
only terms of order $|d|=\sum_{\iota=1}^\mFr d_\iota =\mFr (\mFr-1)$ do not vanish.
This comes from the homogeneous degree of the ratio on the right hand side, which is
$\sum_{\iota} j_\iota - \mFr (\kFr-\mFr + 1)$.
\end{conj}

\begin{exam}
In genus $5$ the highest multiplicity of characteristics is~$3$, this characteristic is unique and corresponds to
partition with $\I_3=\{\}$. Then $\hat{S}[\{\}]$ is a
constant symmetric tensor of order~$3$, whose non-vanishing entries are the following
(with all permutations of indices)
\begin{align*}
 &(\hat{S}[\{\}])_{1,3,5} = -1,& &(\hat{S}[\{\}])_{1,4,4} = (\hat{S}[\{\}])_{2,2,5} = 2,&\\
 &(\hat{S}[\{\}])_{2,3,4} = -2,&  &(\hat{S}[\{\}])_{3,3,3} = 6.&
\end{align*}
Again introducing constant $C_5$ given by \eqref{CDef} into \eqref{thomaeK4MFu}, we obtain
\begin{gather}\label{SconstG5}
 \partial_u^3  \theta[\{\}](\omega^{-1} u) \big|_{u=0} = - C_5 \hat{S}[\{\}],
\end{gather}
so $C_5$ is directly computed through
\begin{align*}
 C_5 &= \partial^3_{u_1, u_3,u_5} \theta[\{\}](\omega^{-1} u) \big|_{u=0}
 = -\frac{1}{2} \partial^3_{u_1, u_7,u_7} \theta[\{\}](\omega^{-1} u) \big|_{u=0} \\
 &= -\frac{1}{2} \partial^3_{u_3, u_3,u_9} \theta[\{\}](\omega^{-1} u) \big|_{u=0}
 = \frac{1}{2} \partial^3_{u_3, u_5,u_7} \theta[\{\}](\omega^{-1} u) \big|_{u=0} \\
 &= -\frac{1}{6} \partial^3_{u_5, u_5,u_5} \theta[\{\}](\omega^{-1} u) \big|_{u=0}.
\end{align*}
\end{exam}
\begin{prop}\label{P:Const}
For a hyperelliptic curve of arbitrary genus $g$ with period matrix $\omega$ 
the constant defined by \eqref{CDef} and arising as a normalizing factor in the relation between 
sigma and theta functions \eqref{SigmaDef} is determined by a directional derivative of 
theta function with the characteristic $[\{\}]=[K]$ of maximal multiplicity  as follows
\begin{gather*}
 C_g = \partial^{[(g+1)/2]}_{u_{2(g\modR 2)+1},\dots,u_{2g-7} ,u_{2g-3}} \theta[K](\omega^{-1} u)\big|_{u=0}.
\end{gather*}
\end{prop}

\begin{exam}\label{E:BolzaG6}
In genus $6$  the constant order $3$ symmetric tensor $\hat{S}[\{\}]$ has the following non-vanishing entries, 
cf. genus $5$ case,
\begin{align*}
&(\hat{S}[\{\}])_{2,4,6} = -1,\qquad
(\hat{S}[\{\}])_{2,5,5} = (\hat{S}[\{\}])_{3,3,6} = 2,\\
&(\hat{S}[\{\}])_{3,4,5} = -2,\qquad
(\hat{S}[\{\}])_{4,4,4} = 6.&
\end{align*}
There are $2g+1=13$ characteristics $[\I_3]=[\{\iota\}]$ with
symmetric tensor $\hat{S}[\{\iota\}]$ such that
\begin{align*}
&(\hat{S}[\{\iota\}])_{1,3,5} = -1,\qquad
(\hat{S}[\{\iota\}])_{1,4,4} = (\hat{S}[\{\iota\}])_{2,2,5}  = 2,&\\
&(\hat{S}[\{\iota\}])_{2,3,4} = -2,\qquad (\hat{S}[\{\iota\}])_{3,3,3} = 6,&\\
&(\hat{S}[\{\iota\}])_{1,3,6} = - (\hat{S}[\{\iota\}])_{1,4,5} = -\frac{1}{2}(\hat{S}[\{\iota\}])_{2,2,6}
= (\hat{S}[\{\iota\}])_{2,3,5} \\ 
&\phantom{(\hat{S}[\{\iota\}])_{1,3,6}} = \frac{1}{2}(\hat{S}[\{\iota\}])_{2,4,4} =
-\frac{1}{2} (\hat{S}[\{\iota\}])_{3,3,4} = e_\iota,&\\
&(\hat{S}[\{\iota\}])_{2,3,6} = -(\hat{S}[\{\iota\}])_{1,4,6} = \frac{1}{2}(\hat{S}[\{\iota\}])_{1,5,5} 
=-(\hat{S}[\{\iota\}])_{2,4,5} \\ 
&\phantom{(\hat{S}[\{\iota\}])_{1,3,6}}
= -\frac{1}{2} (\hat{S}[\{\iota\}])_{3,3,5} = \frac{1}{2} (\hat{S}[\{\iota\}])_{3,4,4} = e_\iota^2,&\\
&(\hat{S}[\{\iota\}])_{2,4,6} = - \frac{1}{2}(\hat{S}[\{\iota\}])_{2,5,5} 
= - \frac{1}{2}(\hat{S}[\{\iota\}])_{3,3,6} \\
&\phantom{(\hat{S}[\{\iota\}])_{1,3,6}} =  \frac{1}{2}(\hat{S}[\{\iota\}])_{3,4,5} 
 =  -\frac{1}{6}(\hat{S}[\{\iota\}])_{4,4,4} = e_\iota^3.& 
\end{align*}
There are plenty of possibilities to compute $e_\iota$, one of them is
\begin{align}\label{BolzaI3}
 e_\iota &= -\frac{\partial^3_{u_1, u_5, u_{11}}\theta[\I_3](\omega^{-1} u)}
 {\partial^3_{u_1, u_5, u_9} \theta[\I_3](\omega^{-1} u)}\Big|_{u=0}.
\end{align}
\end{exam}

The above observations \eqref{BolzaI2} and \eqref{BolzaI3} are generalized in
\begin{prop}\label{P:Bolza}
Let $\I_\mFr$ be a set of $g-\kFr$ indices, and $\mFr=[(\kFr+1)/2]$.
Elementary symmetric polynomials in branch points $\{e_i\mid i\in \I_\mFr\}$ of genus $g$ 
hyperelliptic curve with period matrix $\omega$ are defined by
\begin{align*}
 &s_j(\I_\mFr) = (-1)^j \frac{\partial^\mFr_{u_{2\kFr-4(\mFr-1)-1},\dots,u_{2\kFr-5} ,u_{2\kFr+2j-1}}
 \theta[\I_\mFr](\omega^{-1} u)}
 {\partial^\mFr_{u_{2\kFr-4(\mFr-1)-1},\dots,u_{2\kFr-5} ,u_{2\kFr-1}} \theta[\I_\mFr](\omega^{-1} u)}\Big|_{u=0}.
\end{align*} 
In particular,
\begin{align*}
 &e_\iota = - \frac{\partial^{[g/2]}_{u_{2(g\modR 2)+1},\dots,u_{2g-7} ,u_{2g-1}} \theta[\{\iota\}](\omega^{-1} u)}
 {\partial^{[g/2]}_{u_{2(g\modR 2)+1},\dots,u_{2g-7} ,u_{2g-3}} \theta[\{\iota\}](\omega^{-1} u)}\Big|_{u=0}.
\end{align*} 
\end{prop}
This extends Bolza formulas to an arbitrary genus hyperelliptic curve.

\begin{rem}\label{C:DetD3theta}
From \eqref{SmatrDef3} one could observe that order $3$ tensor $\hat{S}(\I_3)$ belongs to the third tensor power
$\mathcal{S}_{5}^{\otimes 3}$ of the vector space $\mathcal{S}_{5}$ spanned by five vectors 
$\mathbf{s}_{0}$, $\mathbf{s}_{1}$, $\mathbf{s}_{2}$, 
$\mathbf{s}_{3}$, $\mathbf{s}_{4}$ 
such that $\mathbf{s}_{d}=\big(s_{j-\kFr+d}(\I_3)\big)_{j=1}^g$,
here $\kFr=5$ or $6$. And tensor rank of $\hat{S}(\I_3)$ is $19$,
since it is spanned of $19$ basis elements, which are 
\begin{align*}
 &\mathbf{s}_{0}\otimes \mathbf{s}_{2}\otimes \mathbf{s}_{4},&
 &\mathbf{s}_{0}\otimes \mathbf{s}_{4}\otimes \mathbf{s}_{2},&
 &\mathbf{s}_{2}\otimes \mathbf{s}_{0}\otimes \mathbf{s}_{4},&\\ 
 &\mathbf{s}_{4}\otimes \mathbf{s}_{0}\otimes \mathbf{s}_{2},&
 &\mathbf{s}_{2}\otimes \mathbf{s}_{4}\otimes \mathbf{s}_{0},& 
 &\mathbf{s}_{4}\otimes \mathbf{s}_{2}\otimes \mathbf{s}_{0},&\\
 &\mathbf{s}_{1}\otimes \mathbf{s}_{2}\otimes \mathbf{s}_{3},&
 &\mathbf{s}_{1}\otimes \mathbf{s}_{3}\otimes \mathbf{s}_{2},&
 &\mathbf{s}_{2}\otimes \mathbf{s}_{1}\otimes \mathbf{s}_{3},&\\ 
 &\mathbf{s}_{3}\otimes \mathbf{s}_{1}\otimes \mathbf{s}_{2},&
 &\mathbf{s}_{2}\otimes \mathbf{s}_{3}\otimes \mathbf{s}_{1},&
 &\mathbf{s}_{3}\otimes \mathbf{s}_{2}\otimes \mathbf{s}_{1},&\\
 &\mathbf{s}_{0}\otimes \mathbf{s}_{3}\otimes \mathbf{s}_{3},&
 &\mathbf{s}_{3}\otimes \mathbf{s}_{0}\otimes \mathbf{s}_{3},&
 &\mathbf{s}_{3}\otimes \mathbf{s}_{3}\otimes \mathbf{s}_{0},&\\
 &\mathbf{s}_{1}\otimes \mathbf{s}_{1}\otimes \mathbf{s}_{4},&
 &\mathbf{s}_{1}\otimes \mathbf{s}_{4}\otimes \mathbf{s}_{1},&
 &\mathbf{s}_{4}\otimes \mathbf{s}_{1}\otimes \mathbf{s}_{1},&\\
 &\mathbf{s}_{2}\otimes \mathbf{s}_{2}\otimes \mathbf{s}_{2}.
\end{align*} 
Tensor products in the basis are composed in such a way that cumulative weight (in subscripts) is $6$.
So one could find the basis which spans $\hat{S}(\I_3)$ from partitions of $6$ of length $3$ 
formed from $\{0,1,2,3,4\}$.
\end{rem}

\begin{conj}\label{C:DetDHtheta}
With a characteristic $[\I_\mFr]$ of multiplicity $\mFr$ corresponding to a partition
$\I_\mFr\cup \J_\mFr$ with $\I_\mFr=\{i_1$, \ldots, $i_{g-\kFr}\}$ and $\J_\mFr = \{j_1$, \ldots, $j_{g+1+\kFr}\}$, 
 where $\kFr=2\mFr-1$ or $2\mFr$, of indices of $2g+1$ finite branch points the following holds
\begin{gather}\label{thomaeKHMFu}
  \partial_u^\mFr  \theta[\I_\mFr](\omega^{-1} u) \big|_{u=0} 
 = \epsilon \bigg(\frac{\det \omega}{\pi^g}\bigg)^{1/2} \Delta(\I_\mFr)^{1/4} \Delta(\J_\mFr)^{1/4} \hat{S}[\I_\mFr],
\end{gather}
where $u$ are non-normalized variables, and 
order $\mFr$ tensor $\hat{S}[\I_\mFr]$  belongs to the $\mFr$-th tensor power
$\mathcal{S}_{2\mFr-1}^{\otimes \mFr}$ of the vector space $\mathcal{S}_{2\mFr-1}$ spanned by $2\mFr-1$ vectors 
$\mathbf{s}_{0}$, $\mathbf{s}_{1}$, \ldots, $\mathbf{s}_{2\mFr-2}$ 
such that $\mathbf{s}_{d}=\big(s_{j-\kFr+d}(\I_\mFr)\big)_{j=1}^g$. 
The basis spanning $\hat{S}(\I_\mFr)$ could be found from partitions of $\mFr (\mFr-1)$ of length $\mFr$ 
formed from numbers $\{0,1,\dots,2\mFr-2\}$.
\end{conj}

\section{Conclusion and discussion}
The main Theorem~\ref{T:ThN} extends the result of second Thomae theorem to 
derivative theta constants of arbitrary order, namely it gives an expression in terms of period matrix $\omega$
and symmetric functions of branch points for the lowest non-vanishing derivative at $v=0$ 
of theta function with characteristic of arbitrary multiplicity $m$ in hyperelliptic case.
Formula \eqref{thomaeN}, which is called general Thomae formula, 
provides a natural \emph{generalization of second Thomae formula},
and includes, as particular cases, second and first Thomae formulas.

This result gives wide scope for producing further representations of derivative theta constants, and
relations between theta functions. In the present paper the simplest implication is derived, 
which comes from simplification of sums of products of symmetric polynomials in~\eqref{thomaeN}.
Nevertheless, this leads to an essential result. 

In the case of multiplicity $2$ the lowest non-vanishing derivative of theta function is second,
they are naturally arranged in matrices of second derivative theta constants (Hessians). 
All of them have a representation of the form $\omega^t \hat{S}[\I_2] \omega$, and symmetric $g\times g$ 
matrix $\hat{S}[\I_2]$ consists of symmetric polynomials in branch points with indices from~$\I_2$
(Corollary~\ref{C:thomaeK34}). An essential result is the following:
\emph{the rank of $\hat{S}[\I_2]$ is three in any genus} (Theorem~\ref{T:DetD2theta}). 
In connection to this result I refer to the main theorem (Theorem 10) from \cite{GM2006}, 
which is formulated somewhat incorrect. Perhaps, the authors meant that Hessian is singular (not zero), 
that is its determinant is zero, 
since in the proof they assert that the rank of the Hessian is three.  
With this correction Theorem~10 tells about the unique in genus $4$ 
theta function with characteristic $[\{\}]$, which is even and vanishes at $v=0$. 
Theorem~\ref{T:DetD2theta} from the present paper
confirms that in hyperelliptic case the rank of the Hessian of this theta function at $v=0$ is three, 
so the determinant of the Hessian vanishes. Moreover, Theorem~\ref{T:DetD2theta} also extends this result 
to all hyperelliptic curves of higher genera, thus gives the answer to the question posed at the end of \cite{GM2006}
for this class of curves.

Similar representation is obtained in the case of multiplicity $3$ (Corollary~\ref{C:thomaeK56}), where 
third derivative theta constants arranged in $g\times g \times g$ tensors are expressed as a product of three matrices $\omega$
and an order $3$ symmetric tensor $\hat{S}[\I_3]$ whose entries are symmetric polynomials in branch points with indices from~$\I_3$.
Analysis of the structure of tensor $\hat{S}[\I_3]$ allows to make Conjectures~\ref{C:thomaeKH} and \ref{C:DetDHtheta}
about the case of arbitrary multiplicity. 

Also progress is made comparing to \cite{Gra1988}. The result of \cite{Gra1988} 
is extended to an arbitrary genus hyperelliptic curve, namely (i) the constant \eqref{CDef} 
related to the curve, which also serves as a normalizing factor in the relation between 
sigma and theta functions \eqref{SigmaDef}, is expressed in terms of directional derivative 
of theta function with the characteristic of maximal multiplicity (Proposition~\ref{P:Const}),
and (ii) the product of all theta constants is computed through the Vandermonde determinant in all branch points 
and $\det\omega$ (Theorem~\ref{T:EvenThetaProd}).

One more result is a \emph{generalization of Bolza formulas}, see Examples~\ref{E:BolzaG4},
\ref{E:BolzaG5}, and \ref{E:BolzaG6} for the cases of genera $4$, $5$ and~$6$, and Proposition~\ref{P:Bolza}. 
This confirms the result of \cite[Propositions 4.2 p.\,911]{EHKKL2011} and provides 
a representation in lower derivatives than the given in \cite[Propositions 4.3 p.\,911]{EHKKL2011}.
Moreover, Corollaries~\ref{C:thomaeK34}, \ref{C:thomaeK56}, and similar representations for higher multiplicities 
allow to find a complete list of expressions
which generalize Bolza formulas for separate branch points and symmetric polynomials of any number of them.

\section{Aknowledgement}
The problem of obtaining general Thomae formula was posed by Y.\;Kopeliovich, who also encouraged the author to work on it.
The author is grateful to Y.\;Kopeliovich and V.\;Enolski for fruitful discussions.


\end{document}